\documentclass[a4paper, 11pt, final]{article}

\usepackage{amsfonts,enumerate,amsmath}
\usepackage{graphicx,psfrag}

\allowdisplaybreaks[4]

\newtheorem{proposition}{Proposition}[section]
\newtheorem{theorem}[proposition]{Theorem}
\newtheorem{lemma}[proposition]{Lemma}
\newtheorem{corollary}[proposition]{Corollary}
\newtheorem{definition}[proposition]{Definition}
\newtheorem{remark}[proposition]{Remark}

\newenvironment{proof}{\smallskip\noindent\emph{\textbf{Proof.}}\hspace{1pt}}%
{\hspace{-5pt}{\nobreak\quad\nobreak\hfill\nobreak$\square$\vspace{8pt}%
    \par}\smallskip\goodbreak}

\newenvironment{proofof}[1]{\smallskip\noindent{\textbf{Proof of #1.}}%
\hspace{1pt}}{\hspace{-5pt}{\nobreak\quad\nobreak\hfill\nobreak%
$\square$\vspace{8pt}\par}\smallskip\goodbreak}

\newcommand{\Section}[1]{\section{#1}\setcounter{equation}{0}}

\newcommand{\Id}{\mathinner{\mathrm{Id}}}
\newcommand{\comp}{\mathop\bigcirc}
\newcommand{\piu}[1]{{{[\![\,{{#1}}\,]\!]_{\strut +}^{\vphantom{\strut +}}}}}
\newcommand{\meno}[1]{{{[\![\,{{#1}}\,]\!]_{\strut -}^{\vphantom{\strut +}}}}}

\newcommand{\Lloc}[1]{\mathbf{L^{#1}_{loc}}}
\newcommand{\C}[1]{\mathbf{C^{#1}}}
\newcommand{\PC}{\mathbf{PC}}

\newcommand{\modulo}[1]{{\left|#1\right|}}
\newcommand{\norma}[1]{{\left\|#1\right\|}}
\newcommand{\reali}{{\mathbb{R}}}
\newcommand{\interi}{{\mathbb{Z}}}
\newcommand{\naturali}{{\mathbb{N}}}
\newcommand{\tv}{\mathrm{TV}}
\newcommand{\BV}{\mathbf{BV}}
\newcommand{\Lip}{\C{0,1}}

\renewcommand{\epsilon}{\varepsilon}
\renewcommand{\phi}{\varphi}
\renewcommand{\theta}{\vartheta}
\renewcommand{\L}[1]{\mathbf{L^#1}}

\title{On General Balance Laws with Boundary}

\author{Rinaldo M.~Colombo \\ \small Dipartimento di Matematica \\
  \small Universit\`a degli Studi di Brescia \\ \small Via Branze, 38
  \\ \small 25123 Brescia, Italy \\ \texttt{Rinaldo.Colombo@UniBs.it} \\
  \and Graziano Guerra \\ \small Dip.~di Matematica e Applicazioni \\
  \small Universit\`a di Milano -- Bicocca \\ \small Via Bicocca degli
  Arcimboldi, 8 \\ \small 20126 Milano, Italy \\
  \texttt{Graziano.Guerra@UniMiB.it}}

\begin{document}

\maketitle

\begin{abstract}

  \noindent This paper is devoted to general balance laws (with a
  possibly non local source term) with a non-characteristic
  boundary. Basic well posedness results are obtained, trying to
  provide sharp estimates. In particular, bounds tend to blow up as
  the boundary tends to be characteristic. New uniqueness results for
  the solutions to conservation and/or balance laws with boundary are
  also provided.

  \medskip

  \noindent\textit{2000~Mathematics Subject Classification:} 35L50, 35L65.

  \medskip


\end{abstract}

\Section{Introduction}
\label{sec:Intro}

This paper is concerned with initial boundary value problems (IBVP)
for systems of balance laws of the form
\begin{equation}
  \label{eq:BL}
  \left\{
    \begin{array}{l@{\qquad}rcl}
      \partial_t u +\partial_x f(u) = G(u)
      & x & > & \gamma(t)
      \\
      b\left(u \left(t,\gamma(t)\right) \right) = g(t)
      & t & \geq & 0
      \\
      u(0,x) = u_o(x)
      & x & \geq &\gamma(0)
    \end{array}
  \right.
\end{equation}
where $f$ is smooth, $Df$ is strictly hyperbolic, $u_o$ is the initial
datum and $G$ is a possibly non-local source term. The boundary
$\gamma$ is assumed non characteristic, i.e.~$\ell$ characteristics
point outwards and $n-\ell$ inwards. The role of $b$ is that of
letting $n-\ell$ component of $u$ be assigned by the boundary data
$g$. Above and in what follows, we assume that all $\BV$ functions are
right continuous.

Systems belonging to this class were already considered in the
literature. See, for instance, \cite{ColomboGuerra, ColomboGuerra3}
for the case with a non local source but no boundary
and~\cite{ColomboRosini4} for the case of a Temple type $f$.

Examples of physical models that fit into this class are found,
besides in the cited references, also
in~\cite{ColomboGuerraHertySachers}. There, a model describing the
flow of a fluid in a simple pipeline is based on a system essentially
of the form~(\ref{eq:BL}).

As is well known, preliminary to the study of~(\ref{eq:BL}), is that
of the purely convective system
\begin{equation}
  \label{eq:CPHCL}
  \left\{
    \begin{array}{l@{\qquad}rcl}
      \partial_t u +\partial_x f(u) = 0
      & x & > & \gamma(t)
      \\
      b\left(u \left(t,\gamma(t)\right) \right) = g(t)
      & t & \geq & 0
      \\
      u(0,x) = u_o(x)
      & x & \geq & \gamma(0)
    \end{array}
  \right.
\end{equation}
considered, for instance, in~\cite{Amadori1, AmadoriColombo1,
  ColomboRosini4, DonadelloMarson, DuboisLefloch}. Below, we provide
results on~(\ref{eq:CPHCL}) that are not contained in these papers. In
particular, the present estimates explicitly blow up as the boundary
tends to be characteristic. The choice of the Glimm type functionals
on which most of the proof relies is here simplified, compare for
instance~(\ref{eq:DefFun}) below
with~\cite[(2.10)-(2.13)]{DonadelloMarson}
and~(\ref{eq:Phi})--(\ref{eq:Simple})
with~\cite[(3.5)-(3.10)]{DonadelloMarson}.

In the homogeneous case~(\ref{eq:CPHCL}), we also provide a uniqueness
result that has no analogue in the case of Cauchy problems with no
boundary. Indeed, let $u$ solve~(\ref{eq:CPHCL}) and assume a second
boundary $\bar\gamma$ is given, such that $\bar\gamma \geq
\gamma$. Along $\bar\gamma$ assign the trace of $u$ as boundary data,
i.e.~let $\tilde g(t) = b\left( u \left(t, \gamma(t)\right)
\right)$. Then, the solution to
\begin{equation}
  \label{eq:small}
  \left\{
    \begin{array}{l@{\qquad}rcl}
      \partial_t \tilde u +\partial_x f(\tilde u) = 0
      & x & > & \tilde\gamma(t)
      \\
      b\left(\tilde u \left(t,\tilde\gamma(t)\right) \right) = \tilde g(t)
      & t & \geq & 0
      \\
      \tilde u(0,x) = u_o(x)
      & x & \geq & \tilde\gamma(0)
    \end{array}
  \right.
\end{equation}
coincides with the restriction of $u$ to $x \geq \bar\gamma(t)$, see
Proposition~\ref{prop:uniqueness}. We show that an analogous result
may not hold in the case of~(\ref{eq:BL}), see~(\ref{eq:NonUnicita}).

Besides, we also provide a Lipschitz estimate on the process generated
by~(\ref{eq:CPHCL}) that contains also a second order part on a
generic perturbation, see~\emph{\ref{it:Lipschitz}}) in
Theorem~\ref{thm:SRS}. This technical estimate, already known in less
general situations, played a key role in several other results, see
for instance~\cite[Proposition~3.10]{ColomboGuerra}
and~\cite[Remark~4.1]{AmadoriGuerra2002}.

All what we obtain in the case of~(\ref{eq:CPHCL}) is used in the
proof of the results on~(\ref{eq:BL}). In particular, for both
systems, we provide bounds on the total variation of time like
curves. These estimates are optimal in the sense that they blow up as
the boundary tends to be characteristic, see
propositions~\ref{prop:curves} and~\ref{prop:curvesSource}.

The next section is devoted to the homogeneous
problem~(\ref{eq:CPHCL}), while Section~\ref{sec:Source} presents the
results related to~(\ref{eq:BL}). The proofs are deferred to the last
two sections.

\Section{The Purely Convective IBVP}
\label{sec:Convective}

On system~(\ref{eq:CPHCL}) we require the following conditions:

\begin{description}
\item[($\boldsymbol{f}$)] $f \colon \Omega \to \reali^n$ is smooth,
  with $\Omega \subseteq \reali^n$ being open, such that $Df(u)$ is
  strictly hyperbolic for all $u \in \Omega$, each characteristic
  field is either genuinely nonlinear or linearly degenerate.
\end{description}
\noindent Without loss of generality, we may assume that $0 \in
\Omega$ and for all $u$ in $\Omega$, $Df(u)$ admits $n$ real distinct
eigenvalues $\lambda_1(u), \ldots, \lambda_n(u)$, ordered so that
$\lambda_{i-1}(u) < \lambda_i(u)$ for all $u$, with right eigenvectors
$r_1(u), \ldots, r_n(u)$.

\begin{description}
\item[($\boldsymbol{\gamma}$)] $\gamma \in \Lip(\reali^+;\reali)$ and,
  for a fixed positive $c$, $\lambda_{\ell}(u) + c \leq \dot \gamma(t)
  \leq \lambda_{\ell+1}(u) - c$ for a fixed $\ell \in \{1, \ldots,
  n-1\}$ and for all $u \in \Omega$.
\item[($\boldsymbol{b}$)] $b \in \C1(\Omega; \reali^{n-\ell})$ is such
  that $b(0)=0$ and
  \begin{displaymath}
    \det \left[
      Db(0) \, r_{\ell+1}(0) \quad 
      Db(0) \, r_{\ell+2}(0) \quad \cdots \quad
      Db(0) \, r_{n}(0)
    \right] \neq 0 \,.
  \end{displaymath}
\end{description}

\noindent For notational simplicity, we say below that a curve
$\gamma$ is \emph{$\ell$--non-characteristic} if $\gamma \in
\C{0,1}(\reali^+;\reali)$, and for a fixed positive $c$, for all $u
\in \Omega$, $\lambda_{\ell}(u) + c \leq \dot \gamma(t) \leq
\lambda_{\ell+1}(u) - c$. This notion is more restrictive than that of
a \emph{non-resonant} curve, see~\cite[Chapter~14]{DafermosBook}.

We define below the domain
\begin{displaymath}
  \mathbb{D}_\gamma 
  = 
  \left\{(t,x) \in \reali^+ \times \reali \colon x \geq
    \gamma(t) \right\}
\end{displaymath}
and extend to $\left[0, +\infty \right[ \times \reali$ any function
defined on $\mathbb{D}_\gamma$ to vanish outside $\mathbb{D}_\gamma$.

We slightly modify the definition given in~\cite{Goodman} of solution
to~(\ref{eq:CPHCL}) in the non characteristic case, see
also~\cite{Amadori1, AmadoriColombo1, DonadelloMarson}
and~\cite[Definition~2.1]{ColomboRosini4}. Indeed, here we require the
boundary condition to be satisfied by the solution only \emph{almost
  everywhere}. This softening allows for a simpler proof without any
substantial change, since we provide below a full characterization of
this solution, see~\ref{it:semigroup}), \ref{it:Lipschitz}) with
$\omega = 0$ and~\ref{it:tangent}) in Theorem~\ref{thm:SRS}.

\begin{definition}
  \label{def:SolConv}
  Let $T>0$. A map $u = u(t,x)$ is a solution to~(\ref{eq:CPHCL}) if
  \begin{enumerate}
  \item $ u \in \C0 \left([0,T]; \L1( \reali;\reali^n) \right)$ with
    $u(t,x) \in \Omega$ for a.e.~$(t,x) \in \mathbb{D}_\gamma$ and
    $u(t,x) = 0$ otherwise;
  \item $u(0,x) = u_o (x)$ for a.e.~$x \geq \gamma(0)$ and
    $\displaystyle \lim_{x \to 0+} b \left( u(t,x) \right) = g(t)$ for
    a.e.~$t \geq 0$;
  \item for $x > \gamma(t)$, $u$ is a weak entropy solution to
    $\partial_t u +\partial_x f(u) = 0$.
  \end{enumerate}
\end{definition}

\begin{theorem}
  \label{thm:SRS}
  Let the system~(\ref{eq:CPHCL}) satisfy~($\boldsymbol{f}$),
  ($\boldsymbol{b}$), ($\boldsymbol{\gamma}$). Assume also that $g \in
  \BV(\reali^+;\reali^{n-\ell})$ has sufficiently small total
  variation. Then, there exists a family of closed domains
  \begin{displaymath}
    \mathcal{D}_t 
    \subseteq
    \left\{
      u \in  (\L1 \cap \BV) \left( \reali; \Omega \right)
      \colon
      u(x) = 0 \mbox{ for all } x \leq \gamma(t)
    \right\}
  \end{displaymath}
  defined for all $t\geq 0$ and containing all $\L1$ functions with
  sufficiently small total variation that vanish to the left of
  $\gamma(t)$, a constant $L > 0$ and a process
  \begin{displaymath}
    P (t,t_o) \colon \mathcal{D}_{t_o} \to \mathcal{D}_{t_o+t} \,,
    \qquad \hbox{ for all } t_o,t \geq 0,
  \end{displaymath}
  such that
  \begin{enumerate}[1)]
  \item \label{it:semigroup} for all $t_o\geq 0$ and
    $u\in\mathcal{D}_{t_o}$, $P\left(0,t_o\right)u = u$ while for all
    $t,s,t_o\geq 0$ and $u\in\mathcal{D}_{t_o}$,
    $P\left(t+s,t_o\right)u = P\left(t,t_o+s\right)\circ P(s,t_o)u$;
  \item \label{it:Lipschitz} let $\omega$ be an $\L1$ function with
    small total variation, if $(\bar P, \bar{\mathcal{D}}_t)$ are the
    process and the domain corresponding to the boundary
    $\bar\gamma(t)$ and boundary data $\bar g(t)$, then, for any $u
    \in \mathcal{D}_{t_o}$, $v \in \bar{\mathcal{D}}_{t_o'}$, we have
    the following Lipschitz estimate with a second order error term
    accounting for $\omega$:
    \begin{eqnarray*}
      & &
      \norma{P(t,t_o)u - \bar P(t',t_o')v -
        \omega}_{\L1}
      \\
      & \leq &
      L \cdot \bigg\{\norma{u - v - \omega}_{\L1} 
      + 
      \modulo{t-t'}
      +
      \modulo{t_o-t_o'}
      \\
      & &\qquad
      +
      \int_{t_o}^{t_o+t} 
      \norma{g(\tau) - \bar g(\tau)}
      d\tau
      +
      \sup_{\tau\in[t_o,t]} \modulo{\gamma(\tau)-\bar\gamma(\tau)}
      \\
      & &\qquad 
      +
      t\cdot\tv \left(\omega\right)\bigg\};
    \end{eqnarray*}
  \item \label{it:tangent} the tangent vector to $P$ in the sense
    of~\cite[Section~5]{BressanCauchy} is the map $F$ defined
    at~(\ref{eq:local}), i.e.~for all $t_o\geq 0$ and $u \in
    \mathcal{D}_{t_o}$
    \begin{displaymath}
      \lim_{t\to 0+} \frac{1}{t} \norma{F(t,t_o)u - P(t,t_o)u}_{\L1} =0 \,;
    \end{displaymath}

  \item \label{it:solution} for all $u_o \in {\mathcal{D}}_0$, the map
    $u(t,x) = \left( P(t,0) u_o \right) (x)$ defined for $t \in [0,T]$
    and $(t,x) \in \mathbb{D}_\gamma$, solves~(\ref{eq:CPHCL}) in the
    sense of Definition~\ref{def:SolConv}.
  \end{enumerate}
  \noindent $P$ is uniquely characterized by~\ref{it:semigroup}),
  \ref{it:Lipschitz}) with $\omega = 0$ and~\ref{it:tangent}).
\end{theorem}

\noindent The
conditions~\emph{\ref{it:semigroup})}--\emph{\ref{it:tangent})}
constitute what is the natural generalization to the present case of
the definition of \emph{Standard Riemann Semigroup},
see~\cite[Definition~9.1]{BressanLectureNotes}.

Remark that, in general, the Lipschitz constant $L$ blows up as the
the boundary tends to become characteristic, i.e.~as $c \to 0$,
see~(\ref{eq:Csuc}) and the next proposition. Indeed, in the proof of
Theorem~\ref{thm:SRS}, we prove also the following result on the
regularity of the solutions to~(\ref{eq:CPHCL}) along non
characteristic curves.

\begin{proposition}
  \label{prop:curves}
  Fix a positive $T$. Let the system~(\ref{eq:CPHCL}) satisfy the
  assumptions of Theorem~\ref{thm:SRS} and call $u$ the solution
  to~(\ref{eq:CPHCL}) constructed therein. Let $\Gamma_0, \Gamma_1$ be
  $\tilde\ell$--non-characteristic curves, for $\tilde\ell \in \{1,
  \ldots, n-1\}$. Then, there exists a constant $\mathcal{K} > 0$
  independent from $T, u_o,g$ such that
  \begin{displaymath}
    \int_0^{T}\!
    \norma{u\left(t, \Gamma_0(t) \right) - u\left(t, \Gamma_1(t) \right)}
    dt
    \leq
    \frac{\mathcal{K}}{c}
    \left(\tv(u_o) + \tv(g) \right) \norma{\Gamma_1 - \Gamma_0}_{\C0([0,T])}.
  \end{displaymath}
\end{proposition}

A uniqueness property proved in Section~\ref{sec:TechConv} is the
following.

\medskip

\begin{minipage}{0.55\linewidth}
  \begin{proposition}
    \label{prop:uniqueness}
    Let the system~(\ref{eq:CPHCL}) satisfy the same assumptions of
    Theorem~\ref{thm:SRS} and call $u$ the solution
    to~(\ref{eq:CPHCL}) constructed therein. Let $\tilde\gamma \in
    \C{0,1} (\reali^+;\reali)$ be any $\ell$--non-characteristic curve
    satisfying $\tilde\gamma(t) \geq \gamma(t)$ for all $t \geq
    0$. Define $\tilde g(t) = b\left( u(t, \tilde \gamma (t) +
    \right)$. Then, (\ref{eq:small}) also satisfies the assumptions on
    Theorem~\ref{thm:SRS} and the solution $\tilde u$ constructed by
    this Theorem satisfies
    \begin{displaymath}
      \tilde
      u(t,x) = u(t,x)
    \end{displaymath}
    for all $x \geq \tilde\gamma(t)$ and $t\geq 0$.
  \end{proposition}
\end{minipage}
\begin{minipage}{0.44\linewidth}
  \begin{center}
    \begin{psfrags}
      \psfrag{t}{$t$} \psfrag{x}{$x$} \psfrag{gt}{$\tilde \gamma(t)$}
      \psfrag{g}{$\gamma(t)$} \psfrag{ut=u}{$
        \begin{array}{c}
          u(t,x) \\ = \\ \tilde u(t,x)
        \end{array}$}
      \psfrag{u}{$u(t,x)$}
      \includegraphics[width=5cm]{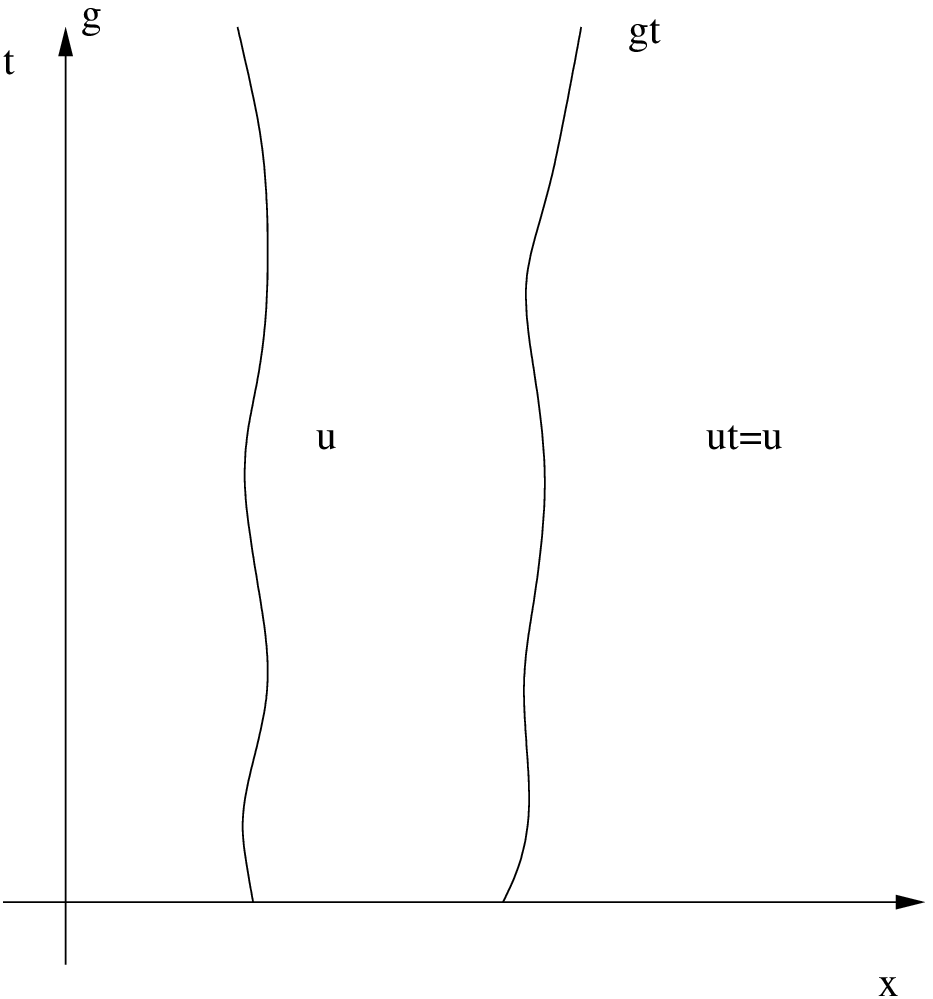}
    \end{psfrags}
  \end{center}

\end{minipage}

\Section{The IBVP with General Source Term}
\label{sec:Source}

To deal with the source term, for all positive $\delta$, define
\begin{displaymath}
  \mathcal{U}_\delta
  =
  \left\{
    u \in \L1 (\reali;\Omega) \colon \tv(u) \leq \delta
  \right\} \,.
\end{displaymath}
We add the following assumption on the source term of
system~(\ref{eq:BL}):
\begin{description}
\item[($\boldsymbol{G}$)] For a positive $\delta_o$, $G \colon
  \mathcal{U}_{\delta_o} \to \L1(\reali,\reali^n)$ is such that for
  suitable positive $L_1,L_2$
  \begin{displaymath}
    \begin{array}{l@{\qquad}rcl}
      \forall\, u,w \in \mathcal{U}_{\delta_o} &
      \norma{G(u) - G(w)}_{\L1} & \leq & L_1 \cdot \norma{u-w}_{\L1}
      \\[5pt]
      \forall\, u \in \mathcal{U}_{\delta_o} &
      \tv \left(G(u) \right) & \leq & L_2\,.
    \end{array}
  \end{displaymath}
\end{description}

\noindent The natural extension of Definition~\ref{def:SolConv} to the
present case is the following.

\begin{definition}
  \label{def:sol}
  Let $T>0$. A map $u = u(t,x)$ is a solution to~(\ref{eq:BL}) if
  \begin{enumerate}
  \item $ u \in \C0 \left([0,T]; \L1( \reali;\reali^n) \right)$ with
    $u(t,x) \in \Omega$ for a.e.~$(t,x) \in \mathbb{D}_\gamma$ and
    $u(t,x)=0$ otherwise;
  \item $u(0,x) = u_o (x)$ for a.e.~$x \geq \gamma(0)$ and
    $\displaystyle \lim_{x \to 0+} b \left( u(t,x) \right) = g(t)$ for
    a.e.~$t \geq 0$;
  \item for $x > \gamma(t)$, $u$ is a weak entropy solution to
    $\partial_t u +\partial_x f(u) = G(u)$.
  \end{enumerate}
\end{definition}

\noindent With this notation, we may now state the extension of
Theorem~\ref{thm:SRS} to the present non homogeneous case.

\begin{theorem}
  \label{thm:main}
  Let system~(\ref{eq:BL}) satisfy~($\boldsymbol{f}$),
  ($\boldsymbol{G}$), ($\boldsymbol{b}$),
  ($\boldsymbol{\gamma}$). Assume also that $g \in
  \BV(\reali^+;\reali^{n-\ell})$ has sufficiently small total
  variation. Then, there exist positive $\delta, L, T$, domains
  $\hat{\mathcal{D}}_t$, for $t \in [0,T]$ and maps
  \begin{displaymath}
    \hat P (t,t_o) \colon 
    \hat{\mathcal{D}}_{t_o} \to \hat{\mathcal{D}}_{t_o+t}
  \end{displaymath}
  for $t_o, t_o+t \in [0,T]$, such that
  \begin{enumerate}[i)]
  \item \label{it:TV} $\hat{\mathcal{D}}_t \supseteq \left\{ u \in
      \mathcal{U}_\delta \colon u(x) = 0 \mbox{ for } x < \gamma(t)
    \right\}$;
  \item \label{it:Process} for all $t_o,t_1,t_2$ with $t_o \in
    \left[0, T\right[$, $t_1 \in \left[0,T-t_o\right[$ and $t_2 \in
    [0, T-t_o-t_1]$, $\hat P(t_2,t_o+t_1) \circ \hat P(t_1,t_o) = \hat
    P(t_1+t_2,t_o)$ and $\hat P(0,t_o) = \Id$;
  \item \label{it:Hard} if $(\bar P, \bar{\mathcal{D}}_t)$ are the
    process and the domains corresponding to the boundary
    $\bar\gamma(t)$ and boundary data $\bar g(t)$, satisfying the same
    assumptions above, then, for $t_o, t_o' \in \left[0, T \right[$,
    $t \in [0, T-t_o]$ and $t' \in [0, T-t_o']$, for all $u \in
    \hat{\mathcal{D}}_{t_o}$, $\bar u \in \hat{\mathcal{D}}_{t_o'}$
    \begin{eqnarray*}
      & &
      \norma{\hat P(t,t_o)u - \bar P(t',t_o')\bar u}_{\L1}
      \\
      & \leq &
      L \cdot \bigg\{\norma{u - \bar u}_{\L1} 
      + 
      \left(1+\norma{u}_{\L1} \right) 
      \left(\modulo{t-t'} + \modulo{t_o-t_o'} \right)
      \\
      & &\qquad
      +
      \int_{t_o}^{t_o+t} 
      \norma{g(\tau) - \bar g(\tau)}
      d\tau
      +
      \sup_{\tau\in[t_o,t]} \modulo{\gamma(\tau)-\bar\gamma(\tau)}
      \bigg\};
    \end{eqnarray*}
  \item \label{it:Tangent} for all $t_o \in \left[0,T\right[$, $t \in
    [0,T-t_o]$, $u \in \hat{\mathcal{D}}_{t_o}$ define
    \begin{displaymath}
      \hat F (t,t_o) u 
      = 
      P(t,t_o)u 
      + 
      t \, G (u) \,
      \chi_{\strut \left[\gamma(t_o+t), +\infty\right[}
    \end{displaymath}
    then
    \begin{displaymath}
      \lim_{t \to 0+}
      \frac{1}{t}
      \norma{\hat P(t,t_o) u - \hat F(t,t_o)u}_{\L1} =0\,;
    \end{displaymath}
  \item \label{it:bdr} for all $u_o \in \hat{\mathcal{D}}_0$, the map
    $u(t,x) = \left( \hat P(t,0) u_o \right) (x)$ defined for $t \in
    [0,T]$ and $(t,x) \in \mathbb{D}_t$, solves~(\ref{eq:BL}) in the
    sense of Definition~\ref{def:sol}.
  \end{enumerate}
  \noindent The process $\hat P$ is uniquely characterized
  by~\ref{it:Process}), \ref{it:Hard}) and~\ref{it:Tangent}).
\end{theorem}

\noindent Again, as remarked after Theorem~\ref{thm:SRS}, the
Lipschitz constant in general blows up as $c \to 0$.  The proof of
this result is deferred to Section~\ref{sec:TechSource}, it heavily
relies on Theorem~\ref{thm:SRS}. Remark that it is possible to extend
to the non homogeneous case also Proposition~\ref{prop:curves}.

\begin{proposition}
  \label{prop:curvesSource}
  Let system~(\ref{eq:BL}) satisfy the same assumptions of
  Theorem~\ref{thm:main} and call $u$ the solution to~(\ref{eq:BL})
  constructed therein. Let $\Gamma_0, \Gamma_1$ be
  $\tilde\ell$--non-characteristic curves, for $\tilde\ell \in \{1,
  \ldots, n-1\}$. Then, for all $u_o$ and $g$, there exists a constant
  $\mathcal{K} > 0$ such that
  \begin{displaymath}
    \int_0^{T}
    \norma{u\left(t, \Gamma_0(t) \right) - u\left(t, \Gamma_1(t) \right)}
    dt
    \leq
    \frac{\mathcal{K}}{c} \, \norma{\Gamma_1 - \Gamma_0}_{\C0([0,T])}.
  \end{displaymath}
\end{proposition}

\begin{remark}
  \label{finalremark}
  Proposition \ref{prop:curvesSource} implies also that, if $\Gamma$
  is any $\ell$--non-characteristic curve, then the map $x \to
  \left(\hat P(t,0)u \right)(x)$ is continuous in
  $x = \Gamma(t)$ for almost all $t \in [0,T]$. Indeed, denote
  $u(t,x) = \left(\hat P(t,0)u\right)(x)$ and compute
  \begin{equation*}
    \begin{split}
      & \int_0^T
      \modulo{u\left(t,\Gamma(t)-\right)-u\left(t,\Gamma(t)\right)} \,
      dt
      \\
      & = \lim_{\epsilon\to 0} \frac{1}{\epsilon} \int_0^\epsilon
      \int_0^T
      \modulo{u\left(t,\Gamma(t)-x\right)-u\left(t,\Gamma(t)\right)}
      \, dt\, dx
      \\
      & \leq \frac{\mathcal{K}}{c} \lim_{\epsilon\to 0}
      \frac{1}{\epsilon}\int_0^\epsilon x\,dx = 0
    \end{split}
  \end{equation*}
\end{remark}

Contrary to Proposition~\ref{prop:curves}, the uniqueness result of
Proposition~\ref{prop:uniqueness} may not be extended to the present
non homogeneous case, due to the non local nature of the source term
here considered. Indeed, let\\
\smallskip
\begin{minipage}{0.44\linewidth}
  \begin{psfrags}
    \psfrag{u12}{\tiny{$ $}}
    \psfrag{u1}{\tiny{$u=1$}}
    \psfrag{u0}{\tiny{$u=0$}}
    \psfrag{0}{$0$} \psfrag{1}{$1$} \psfrag{2}{$2$} \psfrag{3}{$3$}
    \psfrag{4}{$4$} \psfrag{gt}{$\tilde\gamma(t)$}
    \psfrag{g}{$\gamma(t)$} \psfrag{t}{$t$} \psfrag{x}{$x$}
    \psfrag{u}{$u(t,x)$}
    \includegraphics[width=48mm]{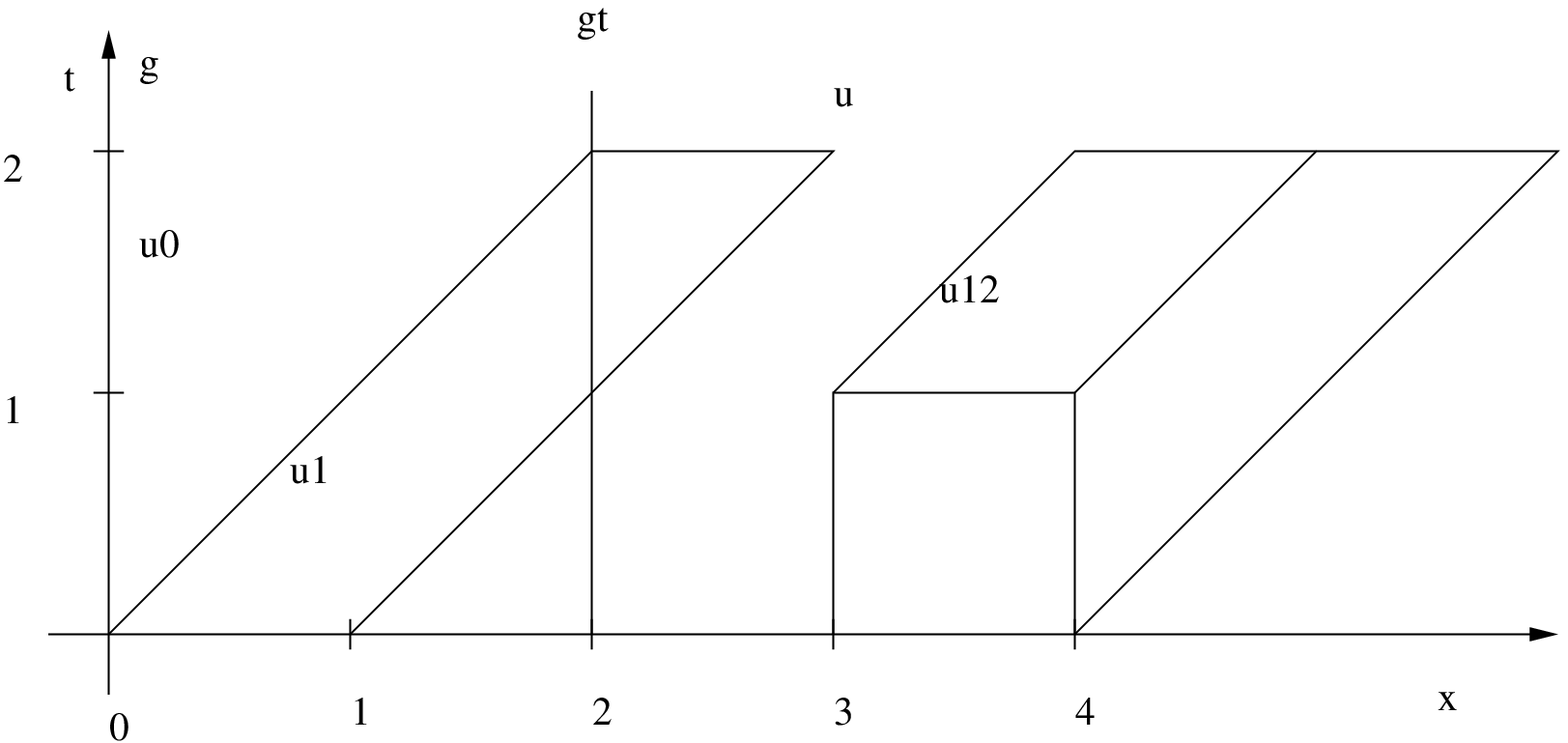}
  \end{psfrags}
\end{minipage}
\begin{minipage}{0.55\linewidth}
  \begin{equation}
    \label{eq:NonUnicita}
    \!\!\!\!\!\!\!\!\!
    \left\{\!\!
      \begin{array}{l}
        \partial_t u + \partial_x u = 
        \left( \int_0^1 u(t,\xi)\, d\xi \right)
        \, \chi_{\strut[3,4]}(x)
        \\
        u(t,0)=0
        \\
        u(0,x)=\chi_{\strut[0,1]}(x) \,.
      \end{array}
    \right.
    \!\!\!\!\!\!
  \end{equation}
\end{minipage}\\
It is immediate to verify that the assumptions of
Theorem~\ref{thm:main} hold. The solution $u$, shown above, is non
zero in the delimited area above and, in particular, for $t \in [0,1]$
and $x \in [3,4]$ but it vanishes for $t \in [0,1]$ and
$x=1$. Therefore, with the same notation of
Proposition~\ref{prop:curves}, letting $\gamma(t) = 0$ and $\tilde
\gamma(t) = 2$ we have $\tilde g(t)=0$ for $t \in
[0,1]$. Problem~(\ref{eq:small}) thus admits, in the present case,
only the trivial solution $u \equiv 0$, contradicting what would be
the analog of Proposition~\ref{prop:curves} in the non homogeneous
case.

\Section{Proofs Related to Section~\ref{sec:Convective}}
\label{sec:TechConv}

Below, $C$ denotes a positive constant dependent only on $f$, $G$ and
$b$ whose precise value is not relevant.

This section is devoted to the homogeneous initial boundary value
problem~(\ref{eq:CPHCL}) and proves Theorem~\ref{thm:SRS}. Our general
reference on the theory of conservation laws
is~\cite{BressanLectureNotes}.

Let $\sigma \to R_j(\sigma)(u)$, respectively $\sigma \to
S_j(\sigma)(u)$, be the $j$-rarefaction curve, respectively the
$j$-shock curve, exiting $u$. If the $j$-th field is linearly
degenerate, then the parameter $\sigma$ above is the arc-length. In
the genuinely nonlinear case,
see~\cite[Definition~5.2]{BressanLectureNotes}, we choose $\sigma$ so
that (see~\cite[formula~(5,37) and~Remark~5.4]{BressanLectureNotes})
\begin{equation}
  \label{eq:strenghtchoice}
  \frac{\partial \lambda_j}{\partial\sigma}
  \left(R_j (\sigma)(u) \right) 
  =
  1
  \quad \mbox{ and } \quad
  \frac{\partial \lambda_j}{\partial\sigma}
  \left(S_j (\sigma)(u) \right) 
  =
  1 \,.
\end{equation}
Introduce the $j$-Lax curve
\begin{displaymath}
  \sigma \to \psi_j (\sigma) (u) =
  \left\{
    \begin{array}{c@{\qquad\mbox{ if }\quad}rcl}
      R_j(\sigma)(u) & \sigma & \geq & 0
      \\
      S_j(\sigma)(u) & \sigma & < & 0
    \end{array}
  \right.
\end{displaymath}
and for $\boldsymbol{\sigma} \equiv (\sigma_1, \ldots, \sigma_n)$,
define the map
\begin{displaymath}
  \boldsymbol{\Psi}(\boldsymbol{\sigma}) 
  =
  \psi_n(\sigma_n) \circ \ldots \circ \psi_1(\sigma_1) 
  \,.
\end{displaymath}
By~\textbf{($\boldsymbol{f}$)},
see~\cite[Paragraph~5.3]{BressanLectureNotes}, given any two states
$u^-,u^+ \in \Omega$ sufficiently close to $0$, there exists a $\C2$
map $E$ such that
\begin{equation}
  \label{eq:E}
  \boldsymbol{\sigma} = E(u^-,u^+)
  \quad \mbox{ if and only if } \quad
  u^+ = \boldsymbol{\Psi}(\boldsymbol{\sigma})(u^-) \,.
\end{equation}
Similarly, let the map $\boldsymbol{S}$ and the vector $\boldsymbol{q}
= ( q_1, \ldots, q_n)$ be defined by
\begin{equation}
  \label{eq:S}
  u^+
  =
  \boldsymbol{S} (\boldsymbol{q})(u^-)
  \quad \mbox{ and } \quad
  \boldsymbol{S} (\boldsymbol{q}) 
  =
  S_n(q_n) \circ \ldots \circ S_1(q_1)
  \,,
\end{equation}
i.e.~$\boldsymbol{S}$ is the gluing of the Rankine - Hugoniot curves.

We first consider the non characteristic Riemann problem at the
boundary
\begin{equation}
  \label{eq:RP}
  \left\{
    \begin{array}{l@{\qquad}rcl}
      \partial_t u +\partial_x f(u) = 0
      & x & > & \gamma(t)
      \\
      b\left(u \left(t,\gamma ( t ) \right) \right) = g_o
      & t & \geq & 0
      \\
      u(0,x) = u_o
      & x & \geq & 0
    \end{array}
  \right.
\end{equation}
where $g_o \in \reali^{n-\ell}$ and $u_o \in \Omega$ are constants and
$\gamma$ satisfies~\textbf{($\boldsymbol{\gamma}$)}. Then, a solution
to~(\ref{eq:RP}) is constructed as in~\cite{Goodman}, see
also~\cite{Amadori1, AmadoriColombo1, ColomboRosini4}.

\begin{lemma}
  \label{lem:RP}
  Let~\textbf{(f)}, \textbf{($\boldsymbol{\gamma}$)} and~\textbf{(b)}
  hold. If $u_o, g_o$ are sufficiently small, then there exists unique
  $E_b^\sigma, E_b^q$ of class $\C2$ and states $u^\sigma, u^q$ such
  that
  \begin{eqnarray*}
    (\sigma_{\ell+1}, \ldots, \sigma_n) =  E_b^\sigma(u_o,g_o)
    & \iff &
    \left\{
      \begin{array}{l}
        b(u^\sigma) = g_o
        \mbox{ and }
        \\
        \psi_n(\sigma_n) \circ \ldots \circ
        \psi_{\ell+1}(\sigma_{\ell+1}) (u^\sigma) = u_o \,,
      \end{array}
    \right.
    \\
    (q_{\ell+1}, \ldots, q_n) = E_b^q(u_o,g_o)
    & \iff &
    \left\{
      \begin{array}{l}
        b(u^q) = g_o
        \mbox{ and }
        \\
        S_n(q_n) \circ \ldots \circ S_{\ell+1}(q_{\ell+1}) (u^q) = u_o \,.
      \end{array}
    \right.
  \end{eqnarray*}
\end{lemma}

\begin{proof}
  We prove this statement only for the Lax curves, the results for the
  shock curves is proved similarly. Let $\sigma_i\to \bar
  \psi_i(\sigma_i)(u)$ be the inverse Lax curve, i.e.
  \begin{displaymath}
    \sigma \to \bar\psi_j (\sigma) (u) =
    \left\{
      \begin{array}{c@{\qquad\mbox{ if }\quad}rcl}
        S_j(\sigma)(u) & \sigma & \geq & 0
        \\
        R_j(\sigma)(u) & \sigma & < & 0.
      \end{array}
    \right.
  \end{displaymath}
  The choice~(\ref{eq:strenghtchoice}) of the parameters implies that
  $\bar \psi_i(-\sigma_i) \circ \psi_i(\sigma_i)(u) = u$ for all small
  $u$ and $\sigma_i$. Define the $\C2$ function
  \begin{displaymath}
    G \left( \sigma_{\ell+1}, \ldots, \sigma_n,g_o,u_o \right)
    =
    b \left(
      \bar\psi_{\ell+1} (-\sigma_{\ell+1}) 
      \circ \ldots \circ
      \bar\psi_{n}(-\sigma_n) (u_o)
    \right) - g_o \,.
  \end{displaymath}
  By~\textbf{($\boldsymbol{b}$)}, $G$ satisfies $G(0,0,0)=0$ and
  \begin{displaymath}
    \det D_{\left(\sigma_{\ell+1},\ldots,\sigma_{n}\right)}G(0,0,0)=
    (-1)^{n-\ell}\det \left[
      Db(0) \, r_{\ell+1}(0) \cdots 
      Db(0) \, r_{n}(0)
    \right] \neq 0 \,.
  \end{displaymath}
  The Implicit Function Theorem guarantees the existence of a map
  $E_b^\sigma = E_b^\sigma (u_o,g_o)$ with the required properties, if
  $(\sigma_{\ell+1}, \ldots, \sigma_n) = E_b^\sigma(u_o,g_o)$ and
  $u^\sigma = \bar\psi_{\ell+1}(-\sigma_{\ell+1}) \circ \ldots \circ
  \bar\psi_{n}(-\sigma_n) (u_o)$.
\end{proof}

The notation introduced above allows the definition of a local flow
tangent to the process generated by~(\ref{eq:CPHCL}). Fix $t_o \geq 0$
and $u \in \mathcal{D}_{t_o}$, define $g_o = g(t_o+)$, $u_o = u\left(
  \gamma(t_o)+\right)$ and $u^\sigma$ as in Lemma~\ref{lem:RP}. Let
\begin{equation}
  \label{eq:tilde}
  \tilde u(x)
  = 
  \left\{
    \begin{array}{l@{\quad\mbox{if }}rcl}
      u^\sigma & x & < & \gamma(t_o)
      \\
      u(x) & x & \geq & \gamma(t_o)
    \end{array}
  \right.
\end{equation}
Call $\mathcal{S}$ the Standard Riemann Semigroup generated by $f$,
see~\cite[Definition~9.1]{BressanLectureNotes}. Finally, for $t \geq
0$, define the tangent vector, see~\cite[Section~5]{BressanCauchy},
\begin{equation}
  \label{eq:local}
  \left( F(t,t_o) u \right) (x) 
  = 
  \left\{
    \begin{array}{l@{\quad\mbox{if }}rcl}
      0 & x & < & \gamma(t_o+t)
      \\
      (\mathcal{S}_t \tilde u)(x) & x & \geq & \gamma(t_o+t)
    \end{array}
  \right.
\end{equation}

We record here the following interaction estimates,
see~Figure~\ref{fig:p}.

\begin{lemma}
  \label{lem:barK}
  Let~\textbf{(f)}, \textbf{($\boldsymbol{\gamma}$)} and~\textbf{(b)}
  hold.  If the following relations hold
  \begin{displaymath}
    \begin{array}{rcl@{\,,\qquad}rcl}
      g^- & = & b(u^-)
      &
      u_r & = &
      \psi_n(\sigma_n) \circ \ldots \circ \psi_1(\sigma_1)(u^-)
      \\
      g^+ & = & b(u^+)
      &    u_r & = &
      \psi_n(\tilde\sigma_n) \circ \ldots \circ 
      \psi_{\ell+1}(\tilde\sigma_{\ell+1})(u^+)
    \end{array}
  \end{displaymath}
  then, we have the estimate
  \begin{displaymath}
    \sum_{i=\ell+1}^n
    \modulo{\tilde \sigma_{i} - \sigma_{i}}
    \leq
    C \left(\sum_{i=1}^\ell \modulo{\sigma_i} + \norma{g^+-g^-} \right) \,.
  \end{displaymath}
  Analogously, for the shock curves, if $\omega$ is a small vector
  satisfying
  \begin{eqnarray*}
    \begin{array}{rcl}
      g & = &b(u),
      \\
      \bar g & = &b(v),
    \end{array}
    \quad \mbox{ and } \quad
    v+\omega =  S_n(q_n) \circ \ldots \circ S_1(q_1)(u)
  \end{eqnarray*}
  then, we have the estimate
  \begin{displaymath}
    \sum_{i=\ell+1}^n
    \modulo{q_{i}}
    \leq 
    C \left(
      \sum_{i=1}^\ell \modulo{q_i}
      +
      \norma{\bar g - g}
      + 
      \norma{\omega}
    \right) \,.
  \end{displaymath}
\end{lemma}
\begin{figure}[htpb]
  \centering
  \begin{psfrags}
    \psfrag{1}{$u^-$} \psfrag{2}{$u_r$} \psfrag{3}{$u^+$}
    \psfrag{4}{$g^-$} \psfrag{5}{$g^+$} \psfrag{6}{$u_o$}
    \psfrag{7}{$u^\sigma$} \psfrag{8}{$g_o$} \psfrag{9}{$(u,v)$}
    \psfrag{a}{$\left(g,\bar g\right)$}
    \includegraphics[width=10cm]{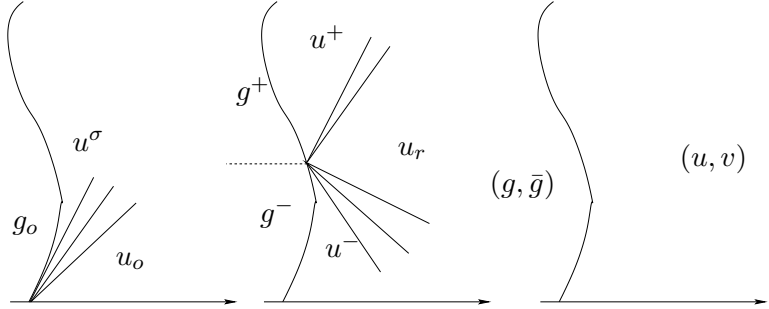}
  \end{psfrags}
  \caption{Interactions at the boundary}
  \label{fig:p}
\end{figure}
\begin{proof}
  By Lemma~\ref{lem:RP}, $(\tilde\sigma_{\ell+1}, \ldots,
  \tilde\sigma_n) = E_b^\sigma (u_r,g^+)$ and $(\sigma_{\ell+1},
  \ldots, \sigma_n)= E_b^\sigma \left(u_r, b\left(
      \psi_\ell(\sigma_\ell) \circ \cdots \circ \psi_1(\sigma_1)
      (u^-)\right) \right)$.  Therefore, the Lipschitz continuity of
  $E_b^\sigma$ implies:
  \begin{eqnarray*}
    \sum_{i=\ell+1}^n
    \modulo{\tilde \sigma_i - \sigma_{i}}
    & \leq &
    C \left(
      \norma{b(u^-) - b\left( \psi_\ell(\sigma_\ell) \circ \cdots
          \circ \psi_1(\sigma_1) (u^-)\right)}
      +
      \norma{g^+-g^-}
    \right)
    \\
    & \leq &
    C \left(
      \sum_{i=1}^\ell \modulo{\sigma_{i}} + \norma{g^+-g^-}
    \right)
  \end{eqnarray*}
  Concerning the shock curves, by Lemma~\ref{lem:RP} we can write
  \begin{eqnarray*}
    (q_{\ell+1}, \ldots, q_n) 
    & = & 
    E_b^q \left(v+\omega
      , b\left(S_\ell(q_\ell) 
        \circ \cdots \circ 
        S_1(q_1)(u)\right)
    \right) \,.
  \end{eqnarray*}
  Since $E_b^q\left(v,\bar g\right)=0$ and $b(u)=g$, the Lipschitz
  continuity of $E_b^q$ and $b$ implies
  \begin{eqnarray*}
    \sum_{\ell+1}^{n} \modulo{q_i}
    & \leq & 
    C \, \norma{
      E_b^q \left(v+\omega
        , b\left(S_\ell(q_\ell) 
          \circ \cdots \circ 
          S_1(q_1)(u)\right)\right)-E_b^q\left(v,\bar g\right)}
    \\
    & \leq & 
    C \left(
      \norma{\omega}
      +
      \norma{b\left(S_\ell(q_\ell) 
          \circ \cdots \circ 
          S_1(q_1)(u)\right)-b(u)} + \norma{g-\bar g}
    \right)
    \\
    &\leq & 
    C\left(
      \sum_{i=1}^{\ell} \modulo{q_i} + \norma{\omega} +
      \norma{\bar g-g}
    \right) \,,
  \end{eqnarray*}
  completing the proof.
\end{proof}

\begin{remark}
  We need below the first statement of Lemma~\ref{lem:barK} in the
  particular case of only one incident wave, i.e.~$u_r =
  \psi_i(\sigma_i)(u^-)$ for some $1 \leq i \leq \ell$.
\end{remark}

We follow the nowadays classical wave front tracking algorithm,
see~\cite{Amadori1, AmadoriColombo1, BressanLectureNotes,
  ColomboRosini4}, to construct solutions to the homogeneous boundary
value problem~(\ref{eq:CPHCL}). Let
$u\in\L1\left(]\gamma(t),+\infty[,\reali^n\right)$ be piecewise
constant with finitely many jumps and assume that $\tv(u)$ is
sufficiently small. Call $J(u)$ the finite set of points where $u$ has
a jump.  Let $\sigma_{x,i}$ be the strength of the $i$-th wave in the
solution of the Riemann problem for
\begin{equation}
  \label{eq:eqHCL}
  \partial_t u + \partial_x f(u) =0
\end{equation}
with data $u(x-)$ and $u(x+)$, i.e.~$(\sigma_{x,1}, \ldots,
\sigma_{x,n}) = E\left( u(x-), u(x +) \right)$. Obviously, if $x \not
\in J(u)$ then $\sigma_{x,i}=0$, for all $i=1,\ldots,n$. In $x =
\gamma(t)$ define
\begin{equation}
  \label{eq:DefSigmaBd}
  \left(\sigma_{\gamma(t),1},\ldots,\sigma_{\gamma(t),n}\right)
  =
  \left(
    0,\ldots,0, E^\sigma_b \left(u(\gamma(t)+),g(t) \right)
  \right) \,.
\end{equation}
Then, consider the Glimm functionals and potentials
\begin{equation}
  \label{eq:DefFunEsatt}
  \begin{array}{rcl}
    \boldsymbol{V}_t(u)
    & = &
    \displaystyle
    K \sum_{x\geq \gamma(t)} \sum_{i=1}^\ell \modulo{\sigma_{x,i}}
    +
    \sum_{x\geq \gamma(t)} \sum_{i=\ell+1}^{n} \modulo{\sigma_{x,i}}
    \\
    \boldsymbol{Q}_t(u)
    & = &
    \displaystyle
    \sum_{(\sigma_{x,i}, \sigma_{y,j}) \in \mathcal{A}}
    \modulo{\sigma_{x,i} \sigma_{y,j}}
    \\
    \boldsymbol{\Upsilon}_t(u)
    & = &
    \displaystyle
    \boldsymbol{V}_t(u)  
    + 
    H_2 \, \boldsymbol{Q}_t(u)
    +
    H_1 \, \tv \left\{g, \left[t,+\infty \right[ \right\}
  \end{array}
\end{equation}
the set $\mathcal{A}$ of approaching waves being defined as usual,
see~\cite{BressanLectureNotes} and the constant $K,H_1,H_2$ to be
defined later.  As in~\cite{ColomboGuerra2}, using
Lemma~\ref{lem:barK}, the Glimm functional $\boldsymbol{\Upsilon}$ can
be extended in a lower semicontinuous way to all functions with small
total variation in $\L1 (\reali; \reali^n)$ that vanish for $x \leq
\gamma(t)$. On the contrary, the interaction potential
$\boldsymbol{Q}$ alone does not admit a lower semicontinuous
extension, due to the presence of the boundary.

We now construct $\epsilon$-approximate solutions to~(\ref{eq:CPHCL})
by means of the classical wave front tracking technique,
see~\cite{BressanLectureNotes} or~\cite{Amadori1, AmadoriColombo1} for
the case with boundary.

Let $\epsilon > 0$ be fixed and approximate the initial and boundary
data in~(\ref{eq:CPHCL}) by means of piecewise constant functions
$u_o^\epsilon$ and $g^\epsilon$ such that
(see~\cite[formula~(3.1)]{ColomboGuerra2})
\begin{equation}
  \label{eq:approxinitialdata}
  \begin{array}{l}
    \norma{u_o^\epsilon - u_o}_{\L1\left(\left[\gamma(0),
          +\infty\right[;\Omega\right)} 
    < \epsilon
    \,, \quad
    \norma{g^\epsilon - g}_{\L\infty(\reali^+,\reali)} < \epsilon
    \\[15pt]
    \modulo{\boldsymbol{\Upsilon}_0(u_o^\epsilon) -
      \boldsymbol{\Upsilon}_0(u_o)} < \epsilon
    \,, \quad
    \tv \left(g^\epsilon,\left]t,+\infty \right[ \right)
    \leq 
    \tv \left( g, \left]t,+\infty\right[ \right) \hbox{ for }t\geq 0\,.
    \!\!\!\!\!\!\!\!\!\!\!\!\!\!\!\!\!\!
  \end{array}
\end{equation}
To proceed beyond time $t=0$, we construct an approximate solution
to~(\ref{eq:CPHCL}) by means of the \emph{Accurate} and
\emph{Simplified} Riemann solvers,
see~\cite[Paragraph~7.2]{BressanLectureNotes}.  Introduce the
threshold parameter $\rho > 0$ to distinguish which Riemann solver is
used at any interaction in $x>\gamma(t)$. Whenever an interaction
occurs at $(t,x)$ with $x > \gamma(t)$, proceed exactly as
in~\cite[Paragraph~7.2]{BressanLectureNotes}. Recall that the former
solver splits new rarefaction waves in fans of wavelets having size at
most $\epsilon$, while the latter yields nonphysical waves. These
waves are assigned to a fictitious $n+1$-th family and their strength
is the Euclidean distance between the states on their sides.

At any interaction involving the boundary, i.e.~when a wave hits the
boundary as well as when the approximated boundary data changes
(see~Figure~\ref{fig:p}, left and center), we use the Accurate solver,
independently from the size of the interaction. As usual, rarefaction
waves are not further split at interactions.

Along an $\epsilon$-approximate solution, for suitable constants $K,
H_1, H_2$ all greater than $1$, introduce the linear and quadratic
potentials and the Glimm functional:
\begin{equation}
  \label{eq:DefFun}
  \begin{array}{rcl}
    V^\epsilon (t)
    & = &
    \displaystyle
    K \sum_{x\geq \gamma(t)} \sum_{i=1}^\ell \modulo{\sigma_{x,i}}
    +
    \sum_{x\geq \gamma(t)} \sum_{i=\ell+1}^{n+1} \modulo{\sigma_{x,i}}
    \\
    V^\epsilon_g (t)
    & = &
    \displaystyle
    \tv \left( g^\epsilon; \left[t, +\infty\right[\right)
    \\[5pt]
    Q^\epsilon (t)
    & = &
    \displaystyle
    \sum_{(\sigma_{x,i}, \sigma_{y,j}) \in \mathcal{A}}
    \modulo{\sigma_{x,i} \sigma_{y,j}}
    \\
    \Upsilon^\epsilon (t)
    & = &
    \displaystyle
    V^\epsilon (t) + H_1 \, V^\epsilon_g (t) + H_2 \,  Q^\epsilon (t)  \,.
  \end{array}
\end{equation}
The potentials just defined differ from the ones defined
in~(\ref{eq:DefFunEsatt}) because the non--physical waves are
accounted for in a different way. They coincide at $t=0$ because of
the absence, at that time, of non-physical waves. They differ of a
quantity proportional to the total size of non physical waves for
$t>0$.

As usual, changing a little the velocities of the waves, we may assume
that no more than two waves $\sigma',\sigma''$ collide at any
interaction point $(\bar t, \bar x)$. When $\bar x > \gamma(\bar t)$,
the usual interaction estimates yield, for a constant $C>0$ dependent
only on $f$ and $\Omega$,
\begin{displaymath}
  \begin{array}{rcl@{\qquad}rcl}
    \Delta V^\epsilon (\bar t) 
    & \leq &
    C \, K \, \modulo{\sigma' \sigma''}
    &
    \Delta V^\epsilon_g (\bar t)
    & = &
    0
    \\
    \Delta Q^\epsilon (\bar t)
    & \leq &
    -\frac{1}{2} \modulo{\sigma' \sigma''}
    &
    \Delta \Upsilon^{\epsilon} (\bar t)
    & \leq &
    -\frac{H_2}{4} \modulo{\sigma' \sigma''} \,,
  \end{array}
\end{displaymath}
as soon as $C K < H_2 / 4$ and $\delta_o$ is sufficiently small.

When a wave $\sigma$ hits the boundary, Lemma~\ref{lem:barK} implies
that
\begin{displaymath}
  \begin{array}{rcl@{\qquad}rcl}
    \Delta V^\epsilon (\bar t) 
    & \leq &
    (C - K) \, \modulo{\sigma}
    &
    \Delta V^\epsilon_g (\bar t)
    & = &
    0
    \\
    \Delta Q^\epsilon (\bar t)
    & \leq &
    C \, \modulo{\sigma} \, V^\epsilon (\bar t -)
    &
    \Delta \Upsilon^{\epsilon} (\bar t)
    & \leq &
    -\frac{K}{2} \modulo{\sigma}
  \end{array}
\end{displaymath}
as soon as $K> 4C$ and $\delta_o < \frac{1}{2H_2}$.

When the boundary data changes, then
\begin{displaymath}
  \begin{array}{rcl@{\qquad}rcl}
    \Delta V^\epsilon (\bar t) 
    & \leq &
    C \modulo{\Delta g^\epsilon (\bar t)}
    &
    \Delta V^\epsilon_g (\bar t)
    & = &
    - \modulo{\Delta g^\epsilon (\bar t)}
    \\
    \Delta Q^\epsilon (\bar t)
    & \leq &
    C \modulo{\Delta g^\epsilon (\bar t)} \, V^\epsilon (\bar t-)
    &
    \Delta \Upsilon^{\epsilon} (\bar t)
    & \leq &
    -\frac{H_1}{3} \modulo{\Delta g^\epsilon (\bar t)}
  \end{array}
\end{displaymath}
as soon as $H_1 > 3C$ and $\delta_o < 1/H_2$.

The above choices are consistent. Indeed, choose first $H_1$ and $K$,
then $H_2$ and finally $\delta_o$.

The wave front tracking approximation can be constructed for all
times, indeed we show that the total number of interaction points is
finite: waves of families $1, \ldots, \ell$ are created only through
the Accurate solver and the use of the Accurate solver in
$\left\{(t,x) \colon t > 0 ,\, x > \gamma(t) \right\}$ leads to a
uniform decrease in $\Upsilon^\epsilon$. Therefore only a finite
number of waves, which can hit the boundary, is present. Since also
the jumps in the boundary are finite, there are at most a finite
number of points in the boundary with outgoing waves. This
observation, together the argument used in the standard case,
see~\cite{BressanLectureNotes}, shows that the total number of
interactions is finite on all the domain $\left\{(t,x) \colon t \geq 0
  ,\, x \geq \gamma(t) \right\}$.

As in~\cite[Paragraph~7.3]{BressanLectureNotes}, the strength of any
rarefaction, respectively nonphysical, wave is smaller than
$C\epsilon$, respectively $C\rho$. This estimate is proved simply
substituting $Q(t)$ in~\cite[formula~(7.65)]{BressanLectureNotes} with
the strictly decreasing functional $\Upsilon^\epsilon(t)$ defined
at~(\ref{eq:DefFun}).

As in the standard case, choosing $\rho$ sufficiently small we prove
that the total size of nonphysical waves is bounded by $\epsilon$. To
this aim, recall the \emph{generation order} of a wave. Waves created
at time $t = 0$, as well as waves originating from jumps in the
boundary data, are assigned order $1$. When two waves interact in the
interior $\left\{(t,x) \colon t > 0 ,\, x > \gamma(t) \right\}$ of the
domain, the usual procedure~\cite[Paragraph~7.3]{BressanLectureNotes}
is followed. When a wave of order $k$ hits the boundary, all the
reflected wave are assigned the same order $k$.

For $k \geq 1$, define
\begin{eqnarray*}
  V^\epsilon_k (t)
  & = &
  K
  \sum \left\{
    \modulo{\sigma} \colon \sigma
    \begin{array}{l}
      \mbox{has order} \geq k
      \\
      \mbox{is of family} \leq \ell
    \end{array}
    \!\!
  \right\}
  + 
  \sum \left\{
    \modulo{\sigma} \colon \sigma
    \begin{array}{l}
      \mbox{has order} \geq k
      \\
      \mbox{is of family} > \ell
    \end{array}
    \!\!
  \right\}
  \\
  Q^\epsilon_k (t)
  & = &
  \sum \left\{
    \modulo{\sigma \sigma'} \colon
    \sigma, \sigma'
    \begin{array}{l}
      \mbox{ are approaching}
      \\
      \mbox{ one of them has order }\geq k
    \end{array}
  \right\} \,.
\end{eqnarray*}
As in~\cite[Paragraph~7.3]{BressanLectureNotes}, for $k \geq 1$, let
$I_k$ denote the set of those interaction times at which the maximal
order of the interacting waves is $k$. $I_0$ denotes the set
containing $t=0$ and all times at which there is a jump in the
boundary data.  On the other hand, $J_k$ is the set of those
interaction times at which a wave of order $k$ hits the boundary. A
careful examinations of the possible interaction yields the following
table for $k\geq 3$:
\begin{displaymath}
  \begin{array}{rcl@{\qquad}rcl}
    \Delta V^\epsilon_k(t) & = & 0
    &
    t & \in & I_0 \cup I_1 \cup \ldots \cup I_{k-2},
    \\
    \Delta V^\epsilon_k(t) + H_2 \, \Delta Q^\epsilon_{k-1}(t) & \leq & 0
    &
    t & \in & I_{k-1} \cup I_{k} \cup \ldots,
    \\
    \Delta V^\epsilon_k(t) & = & 0
    &
    t & \in & J_1 \cup J_2 \cup \ldots \cup J_{k-1},
    \\
    \Delta V^\epsilon_k (t)
    & \leq & 0
    &
    t & \in & J_{k} \cup J_{k+1} \cup \ldots\,.
  \end{array}
\end{displaymath}
Denote the positive, respectively negative, part of a real number by:
$\piu{x} = \max \{0,x\}$, respectively $\meno{x} =
\piu{-x}$. Therefore, similarly
to~\cite[formula~(7.69)]{BressanLectureNotes}, we get for $k\geq 3$:
\begin{eqnarray*}
  V^\epsilon_k (t) 
  & \leq & 
  \sum_{0<s\leq t} \piu{\Delta V^\epsilon_{k} (s)}
  \\
  & \leq &
  H_2 \sum_{0<s\leq t} \meno{\Delta Q^\epsilon_{k-1} (s)}
  \leq 
  H_2 \sum_{0<s\leq t} \piu{\Delta Q^\epsilon_{k-1} (s)}.
\end{eqnarray*}
Now we need to estimate the last sum: $\widetilde Q^\epsilon_k(t) =
\sum_{0<s\leq t} \piu{\Delta Q^\epsilon_{k} (s)}$. Observe that for $k
\geq 3$:
\begin{displaymath}
  \begin{array}{rcl@{\qquad}rcl}
    \Delta Q^\epsilon_k(t)+\Delta \Upsilon^\epsilon(t)
    \cdot V^\epsilon_k(t-) & \leq & 0
    &
    t & \in & I_0 \cup I_1 \cup \ldots \cup I_{k-2},
    \\
    \Delta Q^\epsilon_k(t) + H_2 \, \Delta Q^\epsilon_{k-1}(t)\cdot
    V^\epsilon(t-) 
    & \leq & 0
    &
    t & \in & I_{k-1},
    \\
    \Delta Q^\epsilon_{k}(t) & \leq & 0
    &
    t & \in & I_{k} \cup I_{k+1} \cup \ldots,
    \\
    \Delta Q^\epsilon_k(t) + \Delta \Upsilon^\epsilon(t)
    \cdot V^\epsilon_k(t-) & \leq & 0
    &
    t & \in & J_1 \cup J_2 \cup \ldots \cup J_{k-1},
    \\
    \Delta Q^\epsilon_k(t)+ \Delta V^\epsilon_k(t)
    \cdot V^\epsilon(t-) & \leq & 0
    &
    t & \in & J_{k} \cup J_{k+1} \cup \ldots\, .
    \\
  \end{array}
\end{displaymath}
Hence we can write
\begin{eqnarray*}
  \widetilde Q_k^\epsilon (t) 
  & \leq &  
  \sum_{0<s\leq t}
  \left[
    \meno{\Delta V^\epsilon_k(s)}
    +
    H_2 \, \meno{\Delta Q^\epsilon_{k-1}(s)}
  \right]
  \sup_{0\leq \tau\leq t} V^\epsilon(\tau) 
  \\
  & & 
  +
  \sum_{0<s\leq t} 
  \meno{\Delta \Upsilon^\epsilon(s)} \cdot
  \sup_{0\leq \tau\leq t} V^\epsilon_k(\tau)
  \\
  & \leq &  
  \delta \sum_{0<s\leq t}
  \left[
    \piu{\Delta V^\epsilon_k(s)}
    +
    H_2 \, \piu{\Delta Q^\epsilon_{k-1}(s)}
  \right]
  +
  \Upsilon^\epsilon(0) \, \sup_{0\leq \tau\leq t} V^\epsilon_k(\tau)
  \\
  & \leq & 
  3\delta H_2 \cdot \widetilde Q^\epsilon_{k-1}(t) \,.
\end{eqnarray*}
By induction we obtain
\begin{displaymath}
  \widetilde Q_k^\epsilon (t)
  \leq 
  (3\delta H_2)^{k-2} \, \widetilde Q_2^\epsilon (t)
  \leq 
  (3\delta H_2)^{k-2} \, \delta \,.
\end{displaymath}
Therefore, if $\delta$ is sufficiently small (so that $3 \delta H_2 <
1$), there exists $N_\epsilon > 0$ such that the total size of the
waves of order greater or equal to $N_\epsilon$ is smaller than
$\epsilon$:
\begin{displaymath}
  V^\epsilon_{k}(t)
  \leq 
  H_2 \, \widetilde Q^\epsilon_k(t)
  \leq 
  H_2 \, (3\delta H_2)^{k-2} \, \delta
  \leq 
  \epsilon 
  \,,\quad
  \hbox{ for } k\geq N_\epsilon \,.
\end{displaymath}
Now we observe that the numbers of wave of a given order, is bounded
by a number that depends on $\epsilon$ but not on the threshold
$\rho$: indeed let $M_\epsilon$ the maximum total number of waves that
can be generated in a solution of a Riemann problem inside the domain
or at the boundary. Let $\overline M$ be the sum of the total number
of jumps in the initial data and in the boundary data.  The wave of
first generation are born at $t=0$, at the jumps in the boundary or
when a wave of first generation hit the boundary. Since the waves of
first generation which can hit the boundary (the ones which belong to
the families $i = 1, \ldots, \ell$) are born only at $t = 0$, the
total number of first generation waves is bounded by a constant
$C_\epsilon^1 = \overline M\cdot M_\epsilon + \overline M\cdot
M_\epsilon\cdot M_\epsilon$ not depending on the threshold. Suppose
now that the number of waves of generation lower or equal to $k$ is
bounded by a constant $C_\epsilon^k$ not depending on $\rho$. The
waves of order $k+1$ can be generated only when two waves of lower
order interact, or when a wave of order $k+1$ hit the boundary. Since
the waves of order $k+1$ which can hit the boundary can only be
generated by interaction of waves of lower order, the total number of
generation $k+1$ waves is bounded by $C_{\epsilon}^{k+1} =
C_\epsilon^k \cdot C_\epsilon^k \cdot M_\epsilon + C_\epsilon^k \cdot
C_\epsilon^k \cdot M_\epsilon \cdot M_\epsilon$ which do not depends
on $\rho$.  Fix now $k$ such that $V_k^\epsilon\leq \epsilon$. Hence
also the total strength of non physical waves with order greater or
equal then $k$ is lower than $\varepsilon$. Then observe that the
total number of non physical waves with order less than $k$ is
obviously bounded by $C_\epsilon^k$. Since the strength of any single
non physical wave is bounded by $C\rho$, if we choose the threshold
$\rho$ such that $C \rho \cdot C_\epsilon^k \leq \epsilon$, we have
that the total strength of non physical waves is bounded by
$2\epsilon$.  Finally we observe that if $\gamma$ is any
$\ell$--non-characteristic curve, then $\tv \left( u \left( \cdot,
  \bar\gamma(\cdot)\right) \right)$ is uniformly bounded by a constant time
$\tv(u_o) + \tv(g)$. Indeed, this property is proved following the
techniques in~\cite[Theorem~14.4.2 and formula~(14.5.19)]{DafermosBook}
with our strictly decreasing functional $\Upsilon^\epsilon$.

The following lemma on the regularity of $u$ along non-characteristic
curves is of use in the sequel.

\begin{lemma}
  \label{lem:curves}
  Fix a positive $T$.  Let $u$ be an $\epsilon$-approximate wave front
  tracking solution to~(\ref{eq:CPHCL}). Let $\Gamma_0, \Gamma_1$ be
  $\ell$--non-characteristic curves. Then, there exists a constant
  $\mathcal{K} > 0$ independent from $T$ such that
  \begin{displaymath}
    \int_0^{T} \!
    \norma{u\left(t, \Gamma_0(t) \right) - u\left(t, \Gamma_1(t) \right)}
    dt
    \leq
    \frac{\mathcal{K}}{c}
    \left(\tv(u_o) + \tv(g) \right) \norma{\Gamma_1 - \Gamma_0}_{\C0([0,T])}
  \end{displaymath}
\end{lemma}

\begin{proof}
  Let $\Gamma$ be an $\ell$--non-characteristic curve. Consider a
  perturbation $\eta \in \C{0,1} (\reali^+;\reali)$ with
  $\norma{\eta}_{\C0} + \norma{\dot\eta}_{\L\infty}$ sufficiently
  small. By the above construction of $\epsilon$-solutions, there
  exist times $t_\alpha$ and states $u_\alpha$ such that
  \begin{equation}
    \label{eq:sum}
    u \left( t, \Gamma(t) \right) 
    =
    \sum_\alpha u_\alpha \, \chi_{\left[t_\alpha,
        t_{\alpha+1}\right[}(t)
    \quad \mbox{ and } \quad
    \Gamma(t_\alpha) =
    \lambda_\alpha t_\alpha + x_\alpha
  \end{equation}
  Indeed, here $x = \lambda_\alpha t + x_\alpha$ is the equation of a
  discontinuity line in $u$ crossed by $\Gamma$. If $(t_\alpha,
  x_\alpha)$ is a point of interaction in $u$, then we convene that
  all states attained by $u$ in a neighborhood of $(t_\alpha,
  x_\alpha)$ appear in the sum in~(\ref{eq:sum}), possibly multiplied
  by the characteristic function of the empty interval.

  If $\norma{\eta}_{\C1}$ is sufficiently small, then there exists
  times $t'_\alpha$ such that
  \begin{displaymath}
    u \left( t, \Gamma(t)+\eta(t) \right)
    =
    \sum_\alpha u_\alpha \chi_{\left[t_\alpha', t_{\alpha+1}'\right[}(t)
    \quad\mbox{ and }\quad
    \Gamma(t_\alpha') +\eta(t'_\alpha) = \lambda_\alpha t'_\alpha + x_\alpha \,.
  \end{displaymath}
  Subtracting term by term, we obtain
  \begin{eqnarray*}
    \lambda_\alpha\, (t'_\alpha - t_\alpha)
    & = &
    \left( \Gamma(t'_\alpha) - \Gamma(t_\alpha) \right)
    +
    \eta(t'_\alpha) 
    \\
    & = &
    \int_0^1 
    \dot\Gamma \left( \theta t'_\alpha + (1-\theta) t_\alpha\right)
    \, d\theta \,
    (t'_\alpha - t_\alpha)
    +
    \eta(t'_\alpha) \,.
    \\
    \modulo{t'_\alpha - t_\alpha}
    & = &
    \modulo{\frac{\eta(t'_\alpha)}{\lambda_\alpha - \int_0^1 
        \dot\Gamma \left( \theta t'_\alpha + (1-\theta) t_\alpha\right)
        \, d\theta}}
    \\
    & \leq &
    \frac{\norma{\eta}_{\C0}}{c} \,.
  \end{eqnarray*}
  Therefore,
  \begin{eqnarray}
    \nonumber
    & &
    \int_0^{T}
    \norma{
      u\left(t, \Gamma(t) \right) - 
      u\left(t, \Gamma(t)+\eta(t)\right)
    }
    \, dt
    \\
    \nonumber
    & = &
    \sum_\alpha 
    \norma{u_\alpha - u_{\alpha-1}} \, 
    \modulo{t'_\alpha - t_\alpha}
    \ \leq \ 
    \frac{\norma{\eta}_{\C0}}{c}
    \sum_\alpha 
    \norma{u_\alpha - u_{\alpha-1}}
    \\
    & \leq &
    \frac{\norma{\eta}_{\C0}}{c} 
    \tv \left( u(\cdot, \Gamma(\cdot)\right)
    \label{eq:yes}
    \ \leq \ 
    \mathcal{K}
    \frac{\norma{\eta}_{\C0}}{c} 
    \left( \tv(u_o) + \tv(g) \right)
  \end{eqnarray}
  proving Lipschitz continuity for $\eta$ small. We pass to the
  general case through an interpolation argument. Introduce the map
  \begin{displaymath}
    \psi(\theta) 
    =
    \int_0^{T}
    \norma{ 
      u \left(t, (1-\theta)\Gamma_0(t) + \theta\Gamma_1(t) \right) -
      u \left(t, \Gamma_0(t) \right)
    } \, dt \,.
  \end{displaymath}
  The estimates above prove that the map $\theta \to u\left(\cdot,
    (1-\theta)\Gamma_0(\cdot) + \theta\Gamma_1(\cdot) \right)$ is
  continuous in $\L1$, hence also $\psi$ is continuous and
  by~(\ref{eq:yes}) its upper right Dini derivative satisfies
  \begin{displaymath}
    D^+ \psi(\theta) 
    \leq 
    \mathcal{K}
    \frac{\tv(u_o) + \tv(g)}{c} \norma{\Gamma_1-\Gamma_o}_{\C0}
  \end{displaymath}
  for all $\theta \in [0.1]$. Hence, by the theory of differential
  inequalities,
  \begin{eqnarray*}
    \int_0^{T}
    \norma{u\left(t, \Gamma_1(t) \right) - u\left(t, \Gamma_0(t) \right)}
    dt
    & = &
    \psi(1) - \psi(0)
    \\
    & \leq &
    \mathcal{K}
    \frac{\tv(u_o) + \tv(g)}{c} \norma{\Gamma_1-\Gamma_o}_{\C0([0,T])}
  \end{eqnarray*}
  completing the proof.
\end{proof}

We want now to compare different solutions.  Take two
$\epsilon$--approximate solutions $u$, $v$ corresponding to the two
initial data $u_o$, $v_o$ and the two boundary data $g$ and $\bar
g$. Let $\omega$ be a piecewise constant function with the following
properties: $\omega(t,\cdot)$ is an $\L1$--function with small total
variation, $\omega(t,x)$ has finitely many polygonal lines of
discontinuity and the slope of any discontinuity line is bounded in
absolute value by $\widehat \lambda$.  The function $\omega$ does not
need to have any relation with the conservation law.

Define the functions $w = v + \omega$ and $\boldsymbol{q} \equiv(q_1,
\ldots, q_n)$ implicitly by
\begin{displaymath}
  w(t,x)
  =
  \boldsymbol{S} \left( \boldsymbol{q} (t,x) \right) \left( u(t,x) \right)
\end{displaymath}
with $\boldsymbol{S}$ as in~(\ref{eq:S}). We now consider the
functional
\begin{equation}
  \label{eq:Phi}
  \begin{array}{rcl}
    \displaystyle
    \Phi(u,w)(t)
    & = &
    \displaystyle
    \bar K \sum_{i=1}^\ell \int_{\gamma(t)}^{+\infty} \modulo{q_i(t,x)} \,
    W_i(t,x)\, dx  
    \\
    & &
    \displaystyle
    \quad
    +
    \sum_{i=\ell+1}^n \int_{\gamma(t)}^{+\infty} \modulo{q_i(t,x)} \, W_i(t,x)\, dx
  \end{array}
\end{equation}
where $\bar K$ is a constant to be defined later and the weights $W_i$
are defined setting:
\begin{displaymath}
  W_i(t,x)
  =
  1 
  + 
  \kappa_1 A_i(t,x) 
  + 
  \kappa_2 \left( 
    \Upsilon^\epsilon \left(u(t)\right) + 
    \Upsilon^\epsilon \left(v(t)\right) 
  \right) \,.
\end{displaymath}
The functions $A_i$ are defined as follows. Denote by
$\sigma_{x,\kappa}$ the size of a jump (in $u$ or $v$) located at $x$
of the family $\kappa$ ($\kappa=n+1$ for non physical waves). Recall
that $J(u)$, respectively $J(v)$ denote the sets of all jumps in $u$,
respectively in $v$, for $x > \gamma(t)$, while $\bar J(u)$, $\bar
J(v)$ are the sets of the physical jumps only.

If the $i$-th characteristic field is linearly degenerate, we simply
define
\begin{displaymath}
  A_i(x)
  \doteq  
  \sum \left\{
    \modulo{\sigma_{y,\kappa}}
    \colon
    y\in \bar J(u) \cup \bar J(v)
    \mbox{ and }
    \begin{array}{l}
      y<x, \, i < \kappa \leq n, \mbox{ or}
      \\
      y>x, \, 1 \leq \kappa < i
    \end{array}
  \right\}
\end{displaymath}
On the other hand, if the $i$-th field is genuinely nonlinear, the
definition of $A_i$ will contain an additional term, accounting for
waves in $u$ and in $v$ of the same $i$-th family:
\begin{equation}
  \label{eq:Simple}
  \!\!\!  \!\!\!
  \begin{array}{rcl}
    A_i(x)
    & \doteq & 
    \sum
    \left\{
      \modulo{\sigma_{y,\kappa}}
      \colon
      y\in \bar J(u) \cup \bar J(v)
      \mbox{ and }
      \begin{array}{l}
        y<x, \, i < \kappa \leq n, \mbox{ or}\!
        \\
        y>x, \, 1 \leq \kappa < i
      \end{array}\!\!
    \right\}
    \\
    & &
    + 
    \left\{
      \begin{array}{ll} 
        \displaystyle
        \sum
        \left\{
          \modulo{\sigma_{y,i}}
          \colon
          \begin{array}{l}
            y\in \bar J(u), \, y<x  \mbox{ or}
            \\ 
            y\in \bar J(v), \, y>x
          \end{array}
        \right\}
        \quad &\hbox{if } q_i(x) < 0, 
        \\
        \displaystyle
        \sum
        \left\{
          \modulo{\sigma_{y,i}}
          \colon
          \begin{array}{l}
            y\in \bar J(v), \, y<x  \mbox{ or}
            \\ 
            y\in \bar J(u), \, y>x
          \end{array}
        \right\}
        \quad &\hbox{if } q_i(x) \geq 0.
      \end{array}\!\!
    \right.
  \end{array}
\!\!\!
\end{equation}
Recall that non-physical fronts play no role in the definition of
$A_i$. We remark that the function $\omega$ enters the definition of
$A_i$ only indirectly by influencing the sign of the scalar functions
$q_i$. The constants $\kappa_1$, $\kappa_2$ are the same defined
in~\cite{BressanLectureNotes}. We also recall that, since $\delta_o$
is chosen small enough, the weights satisfy $ 1\leq W_i(t,x) \leq 2 $,
hence for a suitable constant $C_3 > 1$,
\begin{equation}
  \label{eq:equiv}
  \frac{1}{C_3} \, \norma{w(t)-u(t)}_{\L1} 
  \leq  
  \Phi(u,w)(t) 
  \leq 
  C_3 \, \norma{w(t)-u(t)}_{\L1},
\end{equation}
where the $\L1$ norm is taken in the interval $\left] \gamma(t),
  +\infty \right[$.
 
We state now the following theorem.

\begin{proposition}
  \label{prop:theorem44}
  Let the system~(\ref{eq:CPHCL}) satisfy the assumptions of
  Theorem~\ref{thm:SRS}. Then, there exists a constant $\delta \in
  \left]0, \delta_o\right[$ such that, let $u$, $v$, $\omega$, $w$ be
  the functions previously defined, satisfying
  $\Upsilon^\epsilon \left(u(t) \right)$, $\Upsilon^\epsilon\left(v(t)\right)$,
  $\Upsilon^\epsilon\left(\omega(t)\right)$,
  $\Upsilon^\epsilon\left(w(t)\right) \leq \delta$, for any $t\geq 0$,
  then one has
  \begin{eqnarray*}
    \Phi (u, w)(t_2)
    & \leq &
    \displaystyle
    \Phi (u, w)(t_1)
    +    
    C  \epsilon  (t_2 - t_1)
    \\
    & &
    \displaystyle
    + 
    C \int_{t_1}^{t_2}
    \left( 
      \norma{
        b\left(u \left(s, \gamma(s) \right) \right) - 
        b\left(v \left(s, \gamma(s) \right) \right)
      } 
      + 
      \tv\left(\omega(s,\cdot)\right)
    \right) 
    ds .
  \end{eqnarray*}
\end{proposition}

\noindent An immediate consequence of the above result that is useful
below is
\begin{equation}
  \label{eq:tosta}
  \begin{array}{rcl}
    \displaystyle
    \Phi (u, w)(t_2)
    & \leq &
    \displaystyle
    \Phi (u, w)(t_1)
    +    
    C  \epsilon  (t_2 - t_1)
    \\
    & &
    \displaystyle
    + 
    C \int_{t_1}^{t_2}
    \left( 
      \norma{g(s) - \bar g(s)} + \tv\left(\omega(s,\cdot)\right)
    \right) 
    ds .
  \end{array}
\end{equation}

\begin{proofof}{Proposition~\ref{prop:theorem44}}
  In this proof we use the main results obtained
  in~\cite{AmadoriGuerra2002, BressanLectureNotes}.  At each $x$
  define the intermediate states $U_0(x)= u(x)$, $U_1(x)$, $\ldots$,
  $U_n(x) = w(x)$ by setting
  \begin{equation*}
    U_i(x)\doteq 
    S_i \left(q_i(x)\right)
    \circ
    S_{i-1} \left(q_{i-1}(x)\right)
    \circ \cdots \circ
    S_1 \left(q_1(x)\right) \left(u(x)\right) \,.
  \end{equation*}
  Moreover, call
  \begin{displaymath}
    \lambda_i(x) \doteq \lambda_i \left( U_{i-1}(x), U_i(x) \right)
  \end{displaymath}
  the speed of the $i$-shock connecting $U_{i-1}(x)$ with $U_i(x)$.
  For notational convenience, we write $q_i^{y+} \doteq q_i(y+)$,
  $q_i^{y-} \doteq q_i(y-)$ and similarly for $W_i^{y\pm}$,
  $\lambda_i^{y\pm}$.  If $y,\tilde y$ are two consecutive points in
  $J = J(u) \cup J(v) \cup J(\omega)$, then $q_i^{y+} = q_i^{\tilde
    y-}$, $W_i^{y+} = W_i^{\tilde y-}$, $\lambda_i^{y+} =
  \lambda_i^{\tilde y-}$. Therefore, similarly
  to~\cite{BressanLectureNotes, DonadelloMarson}, outside the
  interaction times we can compute:
  \begin{eqnarray*}
    {\frac d {dt}} \Phi (u,w)(t) 
    & = & 
    \bar K \sum_{y\in J}\sum_{i=1}^\ell 
    \left(
      W_i^{y+} \modulo{q_i^{y+}} (\lambda_i^{y+} - \dot x_y) -
      W_i^{y-} \modulo{q_i^{y-}} (\lambda_i^{y-} - \dot x_y)
    \right)
    \\
    & & 
    +
    \sum_{y\in J} \sum_{i=\ell+1}^n \!
    \left(
      W_i^{y+} \modulo{q_i^{y+}} (\lambda_i^{y+} - \dot x_y) -
      W_i^{y-} \modulo{q_i^{y-}} (\lambda_i^{y-} - \dot x_y)
    \right)
    \\
    & & 
    +
    \bar K \sum_{i=1}^\ell
    W_i^{\gamma+} \modulo{q_i^{\gamma+}} (\lambda_i^{\gamma+} - \dot \gamma)
    + 
    \sum_{i=\ell +1}^n
    W_i^{\gamma+} \modulo{q_i^{\gamma+}} (\lambda_i^{\gamma+} - \dot \gamma)
  \end{eqnarray*}
  where $\dot x_y$ is the velocity of the discontinuity at the point
  $y$.  This is because the quantities $q_i$ vanish outside a compact
  set.  For each jump point $y\in J$ and every $i=1,\ldots,n$, define
  \begin{eqnarray*}
    \bar q_i^{y \pm}
    & = &
    \left\{
      \begin{array}{l@{\qquad\mbox{if }}rcl}
        \bar K q_i^{y \pm} & i & \leq & \ell
        \\
        q_i^{y \pm} & i & \geq & \ell + 1
      \end{array}
    \right.
    \\
    E_{y,i}
    & = &
    W_i^{y+} \modulo{\bar q_i^{y+}} (\lambda_i^{y+} - \dot x_y) - 
    W_i^{y-} \modulo{\bar q_i^{y-}} (y_i^{y-} - \dot x_y) \,.
  \end{eqnarray*}
  so that
  \begin{eqnarray*}
    {\frac d  {dt}} \Phi (u, w)(t)
    & = &
    \sum_{y\in J} \sum_{i=1}^n E_{y,i}
    \\
    & &
    +
    \bar K \sum_{i=1}^\ell
    W_i^{\gamma+} \modulo{q_i^{\gamma+}} (\lambda_i^{\gamma+} - \dot \gamma)
    + 
    \sum_{i=\ell +1}^n
    W_i^{\gamma+} \modulo{q_i^{\gamma+}} (\lambda_i^{\gamma+} - \dot \gamma) \,.
  \end{eqnarray*}
  Note that $\bar q^{y\pm}_i$ is a reparametrization of the shock
  curve equivalent to that provided by $q^{y\pm}_i$ and that satisfies
  the key property, see~\cite[Remark~5.4]{BressanLectureNotes},
  \begin{displaymath}
    \left(S_i(\bar q_i) \circ S_i (-\bar q_i) \right) (u) = u \,.
  \end{displaymath}
  Therefore, the computations in~\cite[Section~4]{AmadoriGuerra2002}
  and~\cite[Chapter~8]{BressanLectureNotes} apply. As
  in~\cite[formula~(4.13)]{AmadoriGuerra2002} we thus obtain
  \begin{displaymath}
    \sum_{y\in J} \sum_{i=1}^n E_{y,i} 
    \leq 
    C \cdot \left(\epsilon + \tv(\omega) \right) \,.
  \end{displaymath}
  Concerning the term on the boundary,
  \textbf{($\boldsymbol{\gamma}$)} implies that if $i \leq \ell$, then
  $\lambda_i^{\gamma+} - \dot \gamma \leq -c$. Moreover,
  $W^{\gamma+}_i \geq 1$. Hence, if
  $g^\epsilon=b\left(u(t,\gamma(t)+\right)$, $\bar
  g^\epsilon=b\left(v(t,\gamma(t)+\right)$, Lemma~\ref{lem:barK}
  implies
  \begin{eqnarray*}
    & &
    \bar K \sum_{i=1}^\ell
    W_i^{\gamma+} \modulo{q_i^{\gamma+}} (\lambda_i^{\gamma+} - \dot \gamma)
    + 
    \sum_{i=\ell +1}^n
    W_i^{\gamma+} \modulo{q_i^{\gamma+}} (\lambda_i^{\gamma+} - \dot \gamma)
    \\
    & \leq &
    -c \bar K \sum_{i=1}^\ell \modulo{q_i^{\gamma+}}
    + 
    C \sum_{i=\ell+1}^n \modulo{q_i^{\gamma+}}
    \\
    & \leq &
    -c \bar K \sum_{i=1}^\ell \modulo{q_i^{\gamma+}}
    + 
    C \sum_{i=1}^\ell \modulo{q_i^{\gamma+}}
    +
    C \left( \norma{g^\epsilon - \bar g^\epsilon} + \norma{\omega^{\gamma+}} \right)
    \\
    & \leq & 
    C \left( 
      \norma{g^\epsilon - \bar g^\epsilon} + \norma{\omega^{\gamma+}}
    \right)
  \end{eqnarray*}
  provided
  \begin{equation}
    \label{eq:Csuc}
    \bar K > C/c
  \end{equation}
  is sufficiently large. Therefore, reinserting the $t$ variable, we
  obtain
  \begin{eqnarray*}
    \frac{d}{dt}\Phi(u,w)(t)
    & \leq & 
    C \left(
      \epsilon +
      \tv \left( \omega(t,\cdot) \right)
      +
      \norma{\omega(t,\gamma(t)+)}
      +
      \norma{g^\epsilon(t)-\bar g^\epsilon(t)}
    \right)
    \\
    & \leq &
    C \left(
      \epsilon 
      +
      \tv \left( \omega(t,\cdot) \right)
      + 
      \norma{
        b\left(u \left(s, \gamma(s) \right) \right) - 
        b\left(v \left(s, \gamma(s) \right) \right)
      }
    \right) \,.
  \end{eqnarray*}
  Then, standard computations
  (see~\cite[Theorem~8.2]{BressanLectureNotes}) show that when an
  interaction occurs, the possible increase in $A_i(x)$ is compensated
  by a decrease in $\Upsilon^\epsilon$. Therefore, the functional
  $\Phi$ is not increasing at interaction times. Hence, integrating
  the previous inequality, we obtain~(\ref{eq:tosta}).
\end{proofof}

\begin{proposition}
  \label{prop:FirstPart}
  Let system~(\ref{eq:CPHCL}) satisfy the assumptions of
  Theorem~\ref{thm:SRS}. Then, there exists a process $P$
  satisfying~\emph{\ref{it:semigroup})} in Theorem~\ref{thm:SRS},
  \emph{3.}~in Definition~\ref{def:SolConv} and moreover, there exists
  a positive $L$ such that for all $u,v,\omega$,
  \begin{equation}
    \label{eq:PartLip}
    \begin{array}{rcl}
      & &
      \displaystyle
      \norma{P(t,t_o)u - \bar P(t',t_o')v -
        \omega}_{\L1}
      \\
      & \leq &
      \displaystyle
      L \cdot \bigg\{\norma{u - v - \omega}_{\L1} 
      + 
      \modulo{t-t'}
      +
      \modulo{t_o-t_o'}
      \\
      & &\qquad
      \displaystyle
      +
      \int_{t_o}^{t_o+t} 
      \norma{g(\tau) - \bar g(\tau)}
      d\tau
      +
      t\cdot\tv \left(\omega\right)\bigg\} \,.
    \end{array}
  \end{equation}
\end{proposition}

\begin{proof}
  Let $\delta>0$ be the constant of
  Proposition~\ref{prop:theorem44}. Define
  \begin{displaymath}
    \mathcal{D}_t
    =
    \left\{ 
      u \in \L1 ( \reali; \Omega)
      \colon 
      u(x) =0 \mbox{ for all } x \leq \gamma(t)
      \mbox{ and }
      \boldsymbol{\Upsilon}_t(u) \leq \delta/2\right\} \,.
  \end{displaymath}
  Fix $u_o\in\mathcal{D}_o$. Approximate the initial and boundary data
  $(u_o,g_o)$ as in~(\ref{eq:approxinitialdata}). Since
  $\Upsilon^\epsilon(0) \leq \boldsymbol{\Upsilon}_0
  \left(u^\epsilon(0, \cdot) \right) \leq \boldsymbol{\Upsilon}_0
  (u_o)+\epsilon < \delta/2 + \epsilon < \delta$, we can construct the
  $\epsilon$-approximate solutions $u^\epsilon(t,x)$. As
  in~\cite[Section 8.3]{BressanLectureNotes} we observe that for
  $0<\epsilon'\leq \epsilon$, the $\epsilon'$-approximate solution is
  also an $\epsilon$-approximate solution.  Therefore, we can
  apply~(\ref{eq:tosta}) with $u^{\epsilon'}$ in place of $v$ with
  $\omega = 0$ and $g = \bar g$. Hence, because of~(\ref{eq:equiv}),
  we obtain
  \begin{displaymath}
    \norma{u^{\epsilon}(t)-u^{\epsilon'}(t)}_{\L1}
    \leq 
    L \cdot \norma{u^{\epsilon}_o - u^{\epsilon'}_o}_{\L1} + \epsilon \cdot t \,.
  \end{displaymath}
  For any $t\geq 0$, $u^\epsilon(t)$ is a Cauchy sequence which
  converges to a function $u(t) \in \L1 (\reali; \reali^n)$ that
  vanishes for $x \leq \gamma(t)$. The potential
  $\Upsilon^\epsilon(t)$ defined on $\epsilon$-approximate solutions
  is non increasing and differs from $\boldsymbol{\Upsilon}_t
  \left(u^\epsilon(t)\right)$ due to nonphysical waves and to the
  different boundary conditions $g^\epsilon$ and $g$. Therefore,
  (\ref{eq:approxinitialdata}) and the lower semicontinuity of the
  total variation and of $\boldsymbol{\Upsilon}_t$ implies that
  $u(t)\in\mathcal{D}_t$ for any $t\geq 0$.  We set
  $P(t,0)u_o=u(t)$. It is obvious that our procedure can start at any
  time $t_o\geq 0$, so we can define $P(t,t_o)u \in
  \mathcal{D}_{t+t_o}$ for any $u \in \mathcal{D}_{t_o}$.

  We want to show now that the map just defined satisfies all the
  properties of Theorem~\ref{thm:SRS}. The Lipschitz continuity $t \to
  P(t,t_o)u$ is satisfied by construction.  If we now consider a
  different initial and boundary data, say $(v,\tilde g)$ and the same
  boundary curve $\gamma$, in general we have a different map $\tilde
  P$. Taking the limit in~(\ref{eq:tosta}) and using~(\ref{eq:equiv})
  for the corresponding $\epsilon$-approximations, we get that for any
  $\L1$ function $\omega$ dependent only on $x$ and with small total
  variation
  \begin{equation}
    \label{eq:PartLipFirst}
    \!\!\!
    \begin{array}{rcl}
      & &
      \displaystyle
      \norma{P(t, t_o)u - \tilde{P} (t, t_o) v - \omega}_{\L1}
      \\
      & \leq &
      \displaystyle
      L \cdot \left\{
        \norma{u - v - \omega}_{\L1}
        +
        \int_{t_o}^{t_o+t}  \norma{g(\tau) - \tilde g(\tau)} \, d\tau
        +
        t \cdot \tv (\omega)
      \right\} \,.
    \end{array}
  \end{equation}
  bounding the dependence from the error term $\omega$ and proving the
  Lipschitz continuity in $g$ and $u$.

  Point~\emph{3.} in Definition~\ref{def:SolConv} is obtained by
  standard methods, see~\cite[Section~7.4]{BressanLectureNotes}.

  Concerning the process property, take $u \in \mathcal{D}_0$ and
  consider its $\epsilon$--approx\-imation $u^\epsilon$. Let $\tilde
  \epsilon \in \left] 0, \epsilon \right[$ and call $\tilde u
  ^{\tilde\epsilon}$ be the $\tilde\epsilon$-approximate solution with
  initial datum $u^\epsilon(t)$ at time $t$. Then, if $s \geq 0$,
  $\tilde u^{\tilde \epsilon}$ is also an $\epsilon$-approximate
  solution in $[t, t+s]$. Therefore, applying~(\ref{eq:equiv})
  and~(\ref{eq:tosta}) in the interval $[t, t+s]$, we obtain
  \begin{eqnarray*}
    & &
    \norma{P(t+s,0)u -P(s,t) \circ P(t,0) u}_{\L1}
    \\
    & = &
    \lim_{\epsilon \to 0}
    \norma{u^\epsilon(t+s) - P(s,t) u^\epsilon(t)}_{\L1}
    \\
    & = &
    \lim_{\epsilon \to 0}
    \lim_{\tilde\epsilon \to 0}
    \norma{u^\epsilon(t+s) - \tilde u^{\tilde\epsilon}(t+s)}_{\L1}
    \\
    & \leq &
    \lim_{\epsilon \to 0}
    \lim_{\tilde \epsilon \to 0}
    C \left(
      \int_t^{t+s} \norma{g^\epsilon(\xi) - g^{\tilde\epsilon}(\xi)} \,
      d\xi +
      \epsilon \, s
    \right)
    \\
    & = &
    0 \,.
  \end{eqnarray*}
  We can repeat the same argument for any initial data $t_o\geq 0$.

  Concerning the dependence on the initial time $t_o$, take $0 \leq
  t_o \leq t_o'$ and $u \in \mathcal{D}_{t_o}$, $u' \in
  \mathcal{D}_{t_o'}$. If $0 \leq t \leq t_o'-t_o$, then obviously
  \begin{displaymath}
    \norma{P(t,t_o)u-P(t,t_o')u'}_{\L1}
    \leq 
    C \, \left( \modulo{t_o-t_o'} + \norma{u-u'}_{\L1} \right) \,.
  \end{displaymath}
  If $t > t_o'-t_o$, the process property implies
  \begin{eqnarray*}
    & &
    \norma{P(t,t_o)u-P(t,t_o')u'}_{\L1}
    \\
    & = &
    \norma{P(t+t_o-t_o',t_o')\circ P(t_o'-t_o,t_o)u-P(t,t_o')u'}_{\L1}
    \\
    & \leq & 
    C \, \norma{P(t_o'-t_o,t_o)u-u'}_{\L1} + C \, \modulo{t_o-t_o'}
    \\
    & \leq & 
    C \, \norma{u-u'}_{\L1} + C \modulo{t_o-t_o'} \,,
  \end{eqnarray*}
  completing the proof of~(\ref{eq:PartLip}).
\end{proof}

The following proposition extends to the present case the key
properties of the Glimm functionals~(\ref{eq:DefFunEsatt}).

\begin{proposition}
  Let system~(\ref{eq:CPHCL}) satisfy the assumptions of
  Theorem~\ref{thm:SRS}. Then, for any $u \in \mathcal{D}_{0}$, the
  map $t \to \boldsymbol{\Upsilon}_t \left( P(t,0)u \right)$ is non
  increasing for $t \geq 0$.
\end{proposition}

\begin{proof}
  Above, we showed that the map $t \to \Upsilon^\epsilon(t)$ decreases
  along $\epsilon$-approximate solutions. The monotonicity of $t \to
  \boldsymbol{\Upsilon}_t \left( P(t,0)u \right)$ follow passing to
  the limit $\epsilon \to 0$, thanks to the lower semicontinuity
  proved in~\cite{BaitiBressan2, ColomboGuerra2},
  to~(\ref{eq:approxinitialdata}) and to the lower semicontinuity of
  the total variation.
\end{proof}

In order to complete the proof of Theorem~\ref{thm:SRS}, we prove
propositions~\ref{prop:curves} and~\ref{prop:uniqueness} together with
an auxiliary lemma.

\begin{proofof}{Proposition~\ref{prop:curves}}
  Let $u^\epsilon$ be an $\epsilon$-approximate wave front tracking
  solution converging to $u$. Since the convergence is also in
  $\Lloc1(\mathbb{D}_\gamma; \reali^n)$, apply Lemma~\ref{lem:curves}
  and Lebesgue Dominated convergence Theorem to obtain:
  \begin{eqnarray*}
    & &
    \int_0^{T}
    \norma{u\left(t, \Gamma_0(t) \right) - u\left(t, \Gamma_1(t) \right)}
    \, dt
    \\
    & = &
    \lim_{\delta \to 0}
    \frac{1}{\delta} \int_0^\delta \int_0^{T} 
    \norma{
      u \left(t, \Gamma_0(t)+x \right) - 
      u \left(t, \Gamma_1(t) +x
      \right)}
    \, dt \, dx
    \\
    & = &
    \lim_{\delta \to 0} \lim_{\epsilon \to 0}
    \frac{1}{\delta} \int_0^\delta \int_0^{T} 
    \norma{
      u^\epsilon \left(t, \Gamma_0(t)+x \right) - 
      u^\epsilon \left(t, \Gamma_1(t) +x
      \right)}
    \, dt \, dx
    \\
    & \leq &
    \mathcal{K} \cdot
    \frac{\tv(u_o) + \tv(g)}{c} \cdot
    \norma{\Gamma_1 - \Gamma_0}_{\C0([0,T])}
  \end{eqnarray*}
  completing the proof.
\end{proofof}

\begin{lemma}
  \label{lemma:lemma46}
  Let $u^\epsilon$ be an $\epsilon$-approximate wave front tracking
  solution to~(\ref{eq:CPHCL}) converging to $u$. Let $\Gamma$ be an
  $\ell$--non-characteristic curve. Then,
  \begin{displaymath}
    u^\epsilon \left(\cdot,\Gamma(\cdot) \right)
    \to
    u\left(\cdot,\Gamma(\cdot) \right)
    \quad \mbox{ in } \quad \Lloc1 (\reali^+;\Omega) \,.
  \end{displaymath}
\end{lemma}

\begin{proof}
  By the convergence of $u^\epsilon$ to $u$ in $\Lloc1
  (\mathbb{D}_\gamma;\Omega)$, there exists a sequence $\epsilon_\nu$
  converging to $0$ such that for a.e.~$x$
  \begin{displaymath}
    u^{\epsilon_\nu} \left( \cdot, \Gamma(\cdot)+x\right) 
    \to
    u \left( \cdot, , \Gamma(\cdot)+x \right)
    \quad \mbox{ in } \quad 
    \Lloc1 (\reali^+, \Omega) \,.
  \end{displaymath}
  Then, for any $T>0$ and for any $x$ for which the convergence above
  holds,
  \begin{eqnarray*}
    & &
    \int_0^T 
    \norma{
      u^{\epsilon_\nu} \left(t, \Gamma(t)\right) -
      u \left(t, \Gamma(t)\right)
    } \, dt
    \\
    & \leq &
    \int_0^T 
    \norma{
      u^{\epsilon_\nu} \left(t, \Gamma(t)\right) -
      u^{\epsilon_\nu} \left(t, \Gamma(t)+x\right)
    } \, dt
    \\
    & &
    \quad +
    \int_0^T 
    \norma{
      u^{\epsilon_\nu} \left(t, \Gamma(t)+x\right) -
      u \left(t, \Gamma(t)+x\right)
    } \, dt
    \\
    & &
    \quad +
    \int_0^T 
    \norma{
      u \left(t, \Gamma(t)+x\right) -
      u \left(t, \Gamma(t)\right)
    } \, dt
    \\
    & \leq &
    2 \mathcal{K} \frac{\tv(u_o) + \tv(g)}{c} \modulo{x}
    +
    \int_0^T 
    \norma{
      u^{\epsilon_n} \left(t, \Gamma(t)+x\right) -
      u \left(t, \Gamma(t)+x\right)
    } \, dt
  \end{eqnarray*}
  where we used Lemma~\ref{lem:curves} and
  Proposition~\ref{prop:curves}. Hence
  \begin{eqnarray*}
    \limsup_{\nu\to+\infty}
    \int_0^T 
    \norma{
      u^{\epsilon_\nu} \left(t, \Gamma(t)\right) -
      u \left(t, \Gamma(t)\right)
    } \, dt
    & \leq &
    C \, \modulo{x}
  \end{eqnarray*}
  and the final estimate follows by the arbitrariness of $x$,
  independently from the sequence $\epsilon_\nu$, thanks to the
  uniqueness of the limit $u$.
\end{proof}

\begin{proofof}{Proposition~\ref{prop:uniqueness}}
  Let $u^\epsilon$, respectively $\tilde u^\epsilon$, be an
  $\epsilon$-approximate wave front tracking solutions
  of~(\ref{eq:CPHCL}), respectively~(\ref{eq:small}). Apply
  Proposition~\ref{prop:theorem44} and use the
  equivalence~(\ref{eq:equiv}) to obtain
  \begin{eqnarray*}
    & &
    \int_{\tilde\gamma(t)}^{+\infty}
    \norma{u^\epsilon(t,x) - \tilde u^\epsilon(t,x)} \, dx
    \\
    & \leq &
    L \cdot \Bigl(
    \int_{\tilde\gamma(0)}^{+\infty}
    \norma{u^\epsilon (0,x) - \tilde u^\epsilon(0,x)} \, dx
    \\
    & & 
    +
    \int_0^t \norma{
      b \left( u^\epsilon \left(s,\tilde\gamma(s) \right) \right) -
      b \left( \tilde u^\epsilon \left(s,\tilde\gamma(s) \right) \right)
    }
    \, ds
    \Bigr)
    +C \epsilon t
  \end{eqnarray*}
  and the limit $\epsilon \to 0$ completes the proof.
\end{proofof}

\begin{proofof}{Theorem~\ref{thm:SRS}}
  To conclude the proof of Theorem~\ref{thm:SRS}, observe first that
  for the $\epsilon$--approximate solutions, we have
  \begin{displaymath}
    \norma{b\left(u^\epsilon(t,\gamma(t)+)\right)-g(t)}\leq \epsilon
  \end{displaymath}
  therefore as $\epsilon\to 0$ Lemma~\ref{lemma:lemma46}
  implies~\emph{\ref{it:solution})}.

  Finally denote by let $P^{(\gamma,g)}$ and
  $\mathcal{D}^{(\gamma,g)}_t$ the process and the domains
  corresponding to the boundary curve and data $(\gamma,g)$. Fix two
  boundary curve and data $(\gamma,g)$, $(\bar\gamma,\bar g)$, two
  initial data $u_o\in\mathcal{D}^{(\gamma,g)}_0$, $\bar
  u_o\in\mathcal{D}^{(\bar\gamma,\bar g)}_0$ and define
  \begin{displaymath}
    \!\!
    \begin{array}{rcl@{\ }rcl}
      \Gamma_0(t) & = &\min \left\{\gamma(t),\bar\gamma(t)\right\},
      &
      \Gamma_1(t) & = &\max\left\{\gamma(t),\bar\gamma(t)\right\}
      \\
      \tilde u_o(x) & = & \left\{
        \begin{array}{lr}
          0 & \hbox{for } x \leq \Gamma_1(0) \\
          u_o(x) & \hbox{for } x > \Gamma_1(0)
        \end{array}
      \right.
      &
      \tilde{\bar{u}}_o(x)& = & \left\{
        \begin{array}{lr}
          0 & \hbox{for } x \leq \Gamma_1(0) \\
          \bar u_o(x) & \hbox{for } x > \Gamma_1(0)
        \end{array}
      \right.
      \\
      \tilde{g}(t) 
      & = &
      b\!\left(\!
        \left[P^{(\gamma,g)}(t,0)u_o\right] \left( \Gamma_1(t) \right)\right)
      &
      \tilde{\bar  g}(t)
      & = &
      b\!\left(\!
        \left[P^{(\bar \gamma,\bar g)}(t,0)\bar
          u_o\right] \left( \Gamma_1(t) \right)
      \right)\!.
    \end{array}
  \end{displaymath}
  By Proposition~\ref{prop:uniqueness} we have for $x>\Gamma_1(t)$:
  \begin{eqnarray*}
    \left[P^{(\gamma,g)}(t,0)u_o\right](x) & = & 
    \left[P^{(\Gamma_1,\tilde g)}(t,0)\tilde u_o\right](x),\\ 
    \left[P^{(\bar \gamma,\bar g)}(t,0)\bar u_o\right](x)
    & = &
    \left[P^{(\Gamma_1,\tilde {\bar g})}(t,0)\tilde {\bar u}_o\right](x).
  \end{eqnarray*}
  Applying the result for the unchanged boundary curve, we get:
  \begin{eqnarray*}
    & &
    \norma{P^{(\gamma,g)}(t,0)u_o-P^{(\bar \gamma,\bar g)}(t,0)\bar u_o}_{\L1}\\
    & = & \int_{\Gamma_0(t)}^{\Gamma_1(t)}
    \norma{P^{(\gamma,g)}(t,0)u_o-P^{(\bar \gamma,\bar g)}(t,0){\bar
        u}_o} \, dx
    \\
    & &
    +\int_{\Gamma_1(t)}^{+\infty}\norma{P^{(\Gamma_1,\tilde
        g)}(t,0)\tilde u_o-P^{(\Gamma_1,\tilde{\bar g})}(t,0)\tilde
      {\bar u}_o}\,dx
    \\
    & \leq &
    C\modulo{\Gamma_1(t)-\Gamma_0(t)}+
    C\int_{0}^{t}\norma{\tilde g(t)-\tilde{\bar
        g}(t)}\,dt+C\norma{\tilde u_o-\tilde {\bar u}_o}_{\L1}
    \\
    & \leq &
    C\norma{\Gamma_1-\Gamma_0}_{\C0}+C\norma{u_o-{\bar u}_o}_{\L1}
    \\
    & &
    + C \int_{0}^{t} \norma{\tilde g(t)-g(t)+g(t)-\bar g(t) + \bar g(t)-\tilde{\bar g}(t)} \,dt
    \\
    & \leq &
    C\norma{\gamma-\bar \gamma}_{\C0}+C\norma{ u_o-{\bar
        u}_o}_{\L1}+C\int_{0}^{t}\norma{g(t)-\bar g(t)}\,dt
    \\
    & &
    +C\int_{0}^{t}\norma{\tilde
      g(t)-g(t)}\,dt+C\int_{0}^{t}\norma{\bar g(t)-\tilde{\bar
        g}(t)}\,dt.
  \end{eqnarray*}
  Finally Proposition~\ref{prop:curves} and the Lipschitz continuity
  of $b$ imply
  \begin{eqnarray*}
    & &
    \int_{0}^{t}\norma{\tilde g(t)-g(t)}\,dt
    \\
    & = &
    \int_{0}^{t}
    \norma{b\left(\left[P^{(\gamma,g)}(t,0)u_o\right]
        \left( \Gamma_1(t) \right) \right)
      -
      b\left(\left[P^{(\gamma,g)}(t,0)u_o\right](\gamma(t))
      \right)} \, dt
    \\
    & \leq &
    C\norma{\Gamma_1-\gamma}_{\C0}\le\norma{\bar\gamma-\gamma}_{\C0}
  \end{eqnarray*}
  completing the proof of~\emph{\ref{it:Lipschitz})}, since the
  computations for $\bar g$ and $\tilde{\bar g}$ are identical.

  We prove now the tangency condition~\emph{\ref{it:tangent})}. Fix
  $t_o \geq 0$, $u\in \mathcal{D}_{t_o}$ and let $F$ be defined
  by~(\ref{eq:local}) and denote by $\tilde P$ the process defined
  above with $g$ replaced by $\tilde g(t_o+t) = b \left( \left( S_{t}
      \tilde u \right) \left(\gamma(t_o+t) \right) \right)$ with
  $\tilde u$ as in~(\ref{eq:tilde}). By
  Proposition~\ref{prop:uniqueness}, $F(t,t_o)u = \tilde
  P(t,t_o)u$. Using~\emph{\ref{it:Lipschitz})}, we have
  \begin{eqnarray*}
    & &
    \frac{1}{t}
    \norma{P(t,t_o)u - F(t,t_o)u}_{\L1}
    \\
    & \leq &
    \frac{L}{t} \int_{t_o}^{t_o+t} \norma{\tilde g(s) - g(s)} \, ds
    \\
    & \leq &
    \frac{L}{t} \int_{t_o}^{t_o+t} \norma{\tilde g(s) - g(t_o+)} \, ds
    +
    \frac{L}{t} \int_{t_o}^{t_o+t} \norma{g(t_o+) - g(s)} \, ds \,.
  \end{eqnarray*}
  The latter term vanishes as $t \to 0$ by the definition of
  $g(t_o+)$. Consider now the former term. Fix a positive and
  sufficiently small $\delta$ so that the curve $\psi(s) = \gamma(s) +
  \delta (s-t_o)$ is $\ell$--non-characteristic. Let $\xi \in [0,1]$.
  \begin{eqnarray*}
    & &
    \frac{1}{t} \int_{t_o}^{t_o+t} \norma{\tilde g(s) - g(t_o+)} \, ds
    \\
    & = &
    \frac{1}{t} \int_{t_o}^{t_o+t} 
    \norma{
      b \left( (S_{s-t_o} \tilde u)  \left( \gamma(s) \right)  \right) - g(t_o+)
    } \, ds
    \\
    & \leq &
    \frac{1}{t}
    \int_{t_o}^{t_o+t} \norma{
      b \left( (S_{s-t_o} \tilde u) \left( \gamma(s) \right) \right) 
      -
      b \left( (S_{s-t_o} \tilde u) \left((1-\xi)\gamma(s) + \xi\psi(s)\right) \right)
    }
    \, ds
    \\
    & &
    +
    \frac{1}{t}
    \int_{t_o}^{t_o+t} \norma{
      b \left( (S_{s-t_o} \tilde u) \left((1-\xi)\gamma(s) + \xi\psi(s)\right) \right)
      -
      g(t_o+)
    }
    \, ds \,.
  \end{eqnarray*}
  By Proposition~\ref{prop:curves}, the first term is bounded by
  \begin{displaymath}
    \frac{C}{t} 
    \norma{\gamma - \left( (1-\xi)\gamma+\xi\psi\right)}_{\C0([t_o,t_o+t])}
    \leq 
    \frac{C}{t} \norma{\gamma - \psi}_{\C0([t_o,t_o+t])}
    \leq C \delta \,.
  \end{displaymath}
  Concerning the latter term, integrate on $\xi$ over $[0,1]$ and
  obtain, with the change of variable $x = (1-\xi)\gamma(s) +
  \xi\psi(s)$ and with $u^\sigma$ as in Lemma~\ref{lem:RP},
  \begin{eqnarray}
    \nonumber
    & &
    \frac{1}{t}
    \int_{t_o}^{t_o+t} \norma{
      b \left( (S_{s-t_o} \tilde u) \left((1-\xi)\gamma(s) + \xi\psi(s)\right) \right)
      -
      g(t_o+)
    }
    \, ds
    \\
    \nonumber
    & = &
    \frac{1}{t}
    \int_{t_o}^{t_o+t} 
    \frac{1}{\psi(s) - \gamma(s)}
    \int_{\gamma(s)}^{\psi(s)}
    \norma{
      b \left((S_{s-t_o} \tilde u)(x) \right) - b(u^\sigma)
    }
    \, dx \, ds
    \\
    \label{eq:this}
    & \leq &
    \frac{C}{t\delta}
    \int_{t_o}^{t_o+t} 
    \frac{1}{s-t_o}
    \int_{\gamma(s)}^{\psi(s)}
    \norma{(S_{s-t_o} \tilde u) (x) - u^\sigma}
    \, dx \, ds
  \end{eqnarray}
  Following~\cite[Section~9.3]{BressanLectureNotes}, let $U^\sharp$ be
  the Lax solution to the Riemann problem
  \begin{displaymath}
    \left\{
      \begin{array}{l}
        \partial_t u + \partial_x f(u) = 0
        \\
        u(0,x) = \left\{
          \begin{array}{l@{\quad\mbox{if }}rcl}
            u^\sigma & x & < & 0
            \\
            u \left( \gamma(t_o)+\right) & x & \geq & 0 \,.
          \end{array}
        \right.
      \end{array}
    \right.
  \end{displaymath}
  By the basic properties of the solutions to Riemann problem and the
  definition of $\psi$, for all $s \in [t_o, t]$ and $x \in
  [\gamma(s), \psi(s) ]$, $U^\sharp(s,x) = u^\sigma$. Then,
  \begin{eqnarray*}
    (\ref{eq:this})
    & = &
    \frac{C}{t\delta}
    \int_{t_o}^{t_o+t} 
    \frac{1}{s-t_o}
    \int_{\gamma(s)}^{\psi(s)}
    \norma{(S_{s-t_o} \tilde u) (x) - U^\sharp(s,x)}
    \, dx \, ds
    \\
    & \leq &
    \frac{C}{t\delta}
    \int_{t_o}^{t_o+t} 
    \frac{1}{s-t_o}
    \int_{\gamma(t_o)-(s-t_o)\hat\lambda}^{\gamma(t_o)+(s-t_o)\hat\lambda}
    \norma{(S_{s-t_o} \tilde u) (x) - U^\sharp(s,x)}
    \, dx \, ds
  \end{eqnarray*}
  By~\cite[formula~(9.16)]{BressanLectureNotes},
  \begin{displaymath}
    \lim_{s\to t_o}
    \frac{1}{s-t_o}
    \int_{\gamma(t_o)-(s-t_o)\hat\lambda}^{\gamma(t_o)+(s-t_o)\hat\lambda}
    \norma{(S_{s-t_o} \tilde u) (x) - U^\sharp(s,x)}
    \, dx
    =
    0
  \end{displaymath}
  so that $\lim_{t\to0} (\ref{eq:this})=0$. Collecting the various
  terms,
  \begin{displaymath}
    \limsup_{t\to 0}
    \frac{1}{t} \int_{t_o}^{t_o+t} \norma{\tilde g(s) - g(t_o+)} \, ds
    \leq 
    C \delta
  \end{displaymath}
  and by the arbitrariness of $\delta$, the tangency
  condition~\emph{\ref{it:tangent})} follows.

  The characterization of $P$ through~\emph{\ref{it:semigroup})},
  \emph{\ref{it:Lipschitz})} with $\omega = 0$
  and~\emph{\ref{it:tangent})} implies its uniqueness through standard
  computations, see for instance~\cite[Section~6,
  Corollary~1]{BressanCauchy}.
\end{proofof}

\Section{The Source Term}
\label{sec:TechSource}
  
This section is devoted to the source term, similarly
to~\cite{AmadoriGuerra2002, ColomboGuerra, ColomboGuerra3,
  DafermosHsiao, TPLiuQuasilinear} but following the general metric
space technique in~\cite{ColomboGuerra4}, applied to $\L1$ equipped
with the $\L1$-distance $d$. The key point is to show that the map
\begin{equation}
  \label{eq:Flow}
  \check F (t,t_o) u 
  = 
  P(t,t_o)u 
  + 
  t \, G \left( P(t,t_o)u\right)
  \chi_{\strut \left[\gamma(t_o+t), +\infty\right[}
\end{equation}
is a local flow in the sense of~\cite[Definition~2.1]{ColomboGuerra4}
on suitable domains and satisfies the assumptions
of~\cite[Theorem~2.6]{ColomboGuerra4}.

Following~\cite[Section~3]{ColomboGuerra2}, we modify the functional
$\Phi$ in~(\ref{eq:Phi}) and define $\boldsymbol{\Phi}_t$ on all
piecewise constant functions, not necessarily $\epsilon$--approximate
solutions. Therefore, the definition of $\boldsymbol{\Phi}_t$ does not
consider nonphysical waves and $\boldsymbol{\Phi}_0=\Phi$ at time
$t=0$.  Consider two piecewise constant functions $u, v \in
\L1\left(]\gamma(t),+\infty[,\reali^n\right)$ with finitely many jumps
and assume that $\tv(u)$ is sufficiently small.

Define $\boldsymbol{q} \equiv(q_1, \ldots, q_n)$ implicitly by
\begin{displaymath}
  v(x)
  =
  \boldsymbol{S} \left( \boldsymbol{q} (x) \right) \left( u(x) \right)
\end{displaymath}
with $\boldsymbol{S}$ as in~(\ref{eq:S}). We now consider the
functional
\begin{displaymath}
  \boldsymbol{\Phi}_t(u,v)
  =
  \bar K \sum_{i=1}^\ell \int_{\gamma(t)}^{+\infty} \modulo{q_i(x)} \,
  \boldsymbol{W}_i(x)\, dx
  +
  \sum_{i=\ell+1}^n 
  \int_{\gamma(t)}^{+\infty} \modulo{q_i(x)} \, \boldsymbol{W}_i(x)\, dx  
\end{displaymath}
where $\bar K$ is defined in the proof of
Proposition~\ref{prop:theorem44} and the weights $\boldsymbol{W}_i$
are defined setting:
\begin{displaymath}
  \boldsymbol{W}_i(x)
  =
  1 
  + 
  \kappa_1 \boldsymbol{A}_i(x) 
  + 
  \kappa_2 \left( 
    \boldsymbol{\Upsilon}_t (u) + 
    \boldsymbol{\Upsilon}_t (v) 
  \right) \,.
\end{displaymath}
The functions $\boldsymbol{A}_i$ are defined as follows. Let
$\sigma_{x,\kappa}$ be the strength of the $\kappa$-th wave in the
solution of the Riemann problem for~(\ref{eq:eqHCL}) in $u$ or $v$
located at $x$ of the family $\kappa$. Differently from the notation
in Section~\ref{sec:TechConv}, $J(u)$, respectively $J(v)$ denote the
sets of all jumps in $u$, respectively in $v$, for $x \geq
\gamma(t)$. Indeed, we let $x=\gamma(t)$ in $J$ as soon as $b \left( u
  \left( \gamma(t)+ \right) \right) \neq g(t)$ and the waves
$\sigma_{\gamma(t),k}$ are defined as in~(\ref{eq:DefSigmaBd}).

If the $i$-th characteristic field is linearly degenerate, we simply
define
\begin{displaymath}
  \boldsymbol{A}_i(x)
  \doteq  
  \sum \left\{
    \modulo{\sigma_{y,\kappa}}
    \colon
    y\in J(u) \cup J(v)
    \mbox{ and }
    \begin{array}{l}
      y<x, \, i < \kappa \leq n, \mbox{ or}
      \\
      y>x, \, 1 \leq \kappa < i
    \end{array}
  \right\}
\end{displaymath}
On the other hand, if the $i$-th field is genuinely nonlinear, the
definition of $\boldsymbol{A}_i$ will contain an additional term,
accounting for waves in $u$ and in $v$ of the same $i$-th family:
\begin{eqnarray*}
  \boldsymbol{A}_i(x)
  & \doteq & 
  \sum
  \left\{
    \modulo{\sigma_{y,\kappa}}
    \colon
    y\in J(u) \cup J(v)
    \mbox{ and }
    \begin{array}{l}
      y<x, \, i < \kappa \leq n, \mbox{ or}
      \\
      y>x, \, 1 \leq \kappa < i
    \end{array}
  \right\}
  \\
  & &
  + 
  \left\{
    \begin{array}{ll} 
      \displaystyle
      \sum
      \left\{
        \modulo{\sigma_{y,i}}
        \colon
        \begin{array}{l}
          y\in J(u), \, y<x  \mbox{ or}
          \\ 
          y\in J(v), \, y>x
        \end{array}
      \right\}
      \qquad &\hbox{if } q_i(x) < 0, 
      \\
      \displaystyle
      \sum
      \left\{
        \modulo{\sigma_{y,i}}
        \colon
        \begin{array}{l}
          y\in J(v), \, y<x  \mbox{ or}
          \\ 
          y\in J(u), \, y>x
        \end{array}
      \right\}
      \qquad &\hbox{if } q_i(x) \geq 0.
    \end{array}
  \right. 
\end{eqnarray*}
The constants $\kappa_1$, $\kappa_2$ are the same defined
in~\cite[Chapter~8]{BressanLectureNotes}. We also recall that, since
$\delta_o$ is chosen small enough, the weights satisfy $ 1\leq
\boldsymbol{W}_i(x) \leq 2 $, hence for a suitable constant $C_3 > 1$
we have
\begin{displaymath}
  \frac{1}{C_3} \, \norma{v-u}_{\L1} 
  \leq  
  \boldsymbol{\Phi}_t(u,v)
  \leq 
  C_3 \, \norma{v-u}_{\L1},
\end{displaymath}
where the $\L1$ norm is taken in the interval $\left] \gamma(t),
  +\infty \right[$.

For a fixed positive $M$, define
\begin{displaymath}
  \hat{\mathcal{D}}_t^M
  =
  \left\{
    u \in \L1(\reali;\Omega) \colon
    \begin{array}{l}
      u(x) = 0 \mbox{ for all } x < \gamma(t)
      \\
      \boldsymbol{\Upsilon}_t(u) \leq \delta - C(T-t)
      \\
      \norma{u}_{\L1} \leq Me^{Ct} + Ct
    \end{array}
  \right\}
\end{displaymath}
with $\boldsymbol{\Upsilon}_t$ defined in~(\ref{eq:DefFunEsatt}), $C,
\delta$ and $T$ to be specified below.

\begin{lemma}
  \label{lem:Stime}
  For all $t_o \in [0,T]$, $t >0$ sufficiently small and $u,\tilde u
  \in {\mathcal D}_{t_o}$,
  \begin{eqnarray}
    \nonumber
    \mathbf{\Upsilon}_{t_o+t} \left( \check F(t,t_o)u \right)
    & \leq &
    \mathbf{\Upsilon}_{t_o}(u) + C \, t
    \\
    \label{eq:Bold}
    \mathbf{\Phi}_{t_o+t} 
    \left( \check F(t,t_o)u , \check F(t,t_o)\tilde u \right)
    & \leq &
    ( 1 + C t) \, 
    \mathbf{\Phi}_{t_o} (u, \tilde u)
    \,.
  \end{eqnarray}
\end{lemma}

The proof is as that of~\cite[Lemma~3.6 and
Corollary~3.7]{ColomboGuerra}, see
also~\cite[Lemma~2.3]{ColomboGuerra3}.

\begin{corollary}
  For $t$ small, $\check F$ in~(\ref{eq:Flow}) satisfies $\check F(t,
  t_o) \hat {\mathcal{D}}_{t_o}^M \subseteq \hat
  {\mathcal{D}}_{t_o+t}^M$.
\end{corollary}

\begin{proof}
  The bound on $\boldsymbol{\Upsilon}_t$ is a direct consequence of
  Lemma~\ref{lem:Stime}. Concerning the estimate on the $\L1$ norm,
  for $u \in \hat{\mathcal{D}}^M_{t_o}$, compute:
  \begin{eqnarray*}
    & &
    \norma{\check F(t,t_o)u}_{\L1}
    \\
    & = &
    \norma{
      P(t,t_o) u 
      + 
      t G \left(P(t,t_o)u \right)\chi_{\strut[\gamma(t_o+t), +\infty[}
    }_{\L1}
    \\
    & \leq &
    \norma{P(t,t_o)u - u}_{\L1}
    +
    \norma{u}_{\L1}
    +
    t \norma{G\left( P(t,t_o)u\right) - G(0)}_{\L1}
    +
    t \norma{G(0)}_{\L1}
    \\
    & \leq &
    Ct+ \norma{u}_{\L1} + C t \left( \norma{u}_{\L1} + Ct \right) + Ct
    \\
    & \leq &
    (1+Ct) \, \norma{u}_{L1} + Ct
    \\
    & \leq &
    (1+Ct) \, \left( M e^{Ct_o} + Ct_o\right) + Ct
    \\
    & \leq &
    M e^{C(t_o+t)} + C(t_o+t)
  \end{eqnarray*}
  hence $\check F(t,t_o)u$ is in $\hat{\mathcal{D}}^M_{t_o+t}$.
\end{proof}

In what follows, relying
on~\cite[Condition~\textbf{(D)}]{ColomboGuerra4}, we consider $\check
F$ as defined on the domains $\hat{\mathcal{D}}_{t_o}^M$ and not on a
single domain, as in~\cite[Definition~2.1]{ColomboGuerra4}.

\begin{proposition}
  \label{prop:M}
  The map $\check F$ defined in~(\ref{eq:Flow}) is $\L1$ Lipschitz
  continuous, satisfies $\check F\left(0,t_o\right)u = u$ for any
  $(t_o,u)\in \left\{(\tau,w) \colon \tau \in [0,T], \ w \in
    \hat{\mathcal{D}}_\tau^M \right\}$ and there exist positive
  $\mathcal{L}$, independent from $M$, such that for $t_o,t_o' \in
  [0,T]$, $t \in [0,T-t_o]$, $t' \in [0, T-t_o']$, $u \in
  \hat{\mathcal{D}}^M_{t_o}$, $u' \in \hat{\mathcal{D}}^M_{t_o'}$
  \begin{displaymath}
    \norma{ \check F(t',t_o')u'- \check F(t,t_o) u}_{\L1}
    \leq
    \mathcal{L} \left(
      \norma{u'-u}_{\L1}
      + 
      \left( 1+ \norma{u}_{\L1} \right) \modulo{t' - t} 
      +
      \modulo{t_o'-t_o}
    \right).
  \end{displaymath}
\end{proposition}

\begin{proof}
  Compute:
  \begin{eqnarray*}
    & &
    \norma{\check F(t', t_o')u' - \check F (t, t_o) u}_{\L1}
    \\
    & \leq &
    \norma{P(t',t_o')u' - P(t,t_o)u}_{\L1}
    +
    \modulo{t'-t} 
    \norma{G \left( P(t',t_o')u' \right) \chi_{\strut[\gamma(t_o'+t'),
        +\infty[}}_{\L1}
    \\
    & &
    +
    t \norma{
      G\left(P(t',t_o')u'\right) \chi_{\strut[\gamma(t_o'+t'), +\infty[}
      -
      G\left(P(t,t_o)u\right) \chi_{\strut[\gamma(t_o+t), +\infty[}
    }_{\L1}
    \\[8pt]
    & \leq &
    \norma{P(t',t_o')u' - P(t,t_o)u}_{\L1}
    \\
    & &
    +
    \modulo{t'-t} 
    \left( 
      \norma{G \left( P(t',t_o')u' - G(0)\right)}_{\L1}
      +
      \norma{G(0)}_{\L1}
    \right)
    \\
    & &
    +
    t \norma{
      G\left(P(t',t_o')u'\right)
      -
      G\left(P(t,t_o)u\right)
    }_{\L1}
    \\
    & &
    + 
    t \norma{
      G\left( P(t,t_o)u\right)
      \left(
        \chi_{\strut[\gamma(t_o'+t'),+\infty[}
        -
        \chi_{\strut[\gamma(t_o+t),+\infty[}
      \right)
    }_{\L1}
    \\[8pt]
    & \leq &
    (1+Ct) \, \norma{P(t',t_o')u' - P(t,t_o)u}_{\L1}
    +
    \modulo{t'-t} \, \norma{G(0)}_{\L1}
    \\
    & &
    +
    C \, \modulo{t'-t} \, \norma{P(t',t_o')u'}_{\L1}
    +
    C \, t \, \modulo{\gamma(t_o'+t') - \gamma(t_o+t)}
    \\[8pt]
    &\leq &
    C \left( \modulo{t'-t} + \modulo{t_o'-t_o} + \norma{u'-u}_{\L1}
    \right)
    \\
    & &
    +
    C \modulo{t'-t} \left( \norma{P(t',t_o')u' - u'}_{\L1} +
      \norma{u'}_{\L1} \right)
    \\[8pt]
    & \leq &
    C \left(
      \left( 1+\norma{u'}_{\L1} \right) \modulo{t'-t} 
      +
      \modulo{t_o'-t_o} + \norma{u'-u}_{\L1}
    \right)
  \end{eqnarray*}
  completing the proof.
\end{proof}

Recall~\cite[Definition~2.3]{ColomboGuerra4}: an Euler
$\epsilon$-polygonal is
\begin{equation}
  \label{eq:polygonal}
  \check F^\epsilon (t,t_o) \, u
  =
  \check F(t-k\epsilon, t_o+k\epsilon) \circ 
  \comp_{h=0}^{k-1} \check F(\epsilon, t_o+h\epsilon) \, u
\end{equation}
for $k=[t/\epsilon]$. Above and in what follows, we denote the
recursive composition $\comp_{i=1}^n f_i = f_1 \circ f_2 \circ \ldots
\circ f_n$. Here, $[\,\cdot\,]$ stands for the integer part, i.e.~for
$s \in \reali$, $[s] = \max \{k \in \interi \colon k \leq s \}$.

The hypotheses to apply~\cite[Theorem~2.6]{ColomboGuerra4} are
satisfied.

\begin{proposition}
  \label{prop:yes}
  The local flow $\check F$ in~(\ref{eq:Flow}) is such that there
  exist
  \begin{enumerate}
  \item \label{it:first} a positive constant $C$ such that for all
    $t_o \in [0,T]$ and all $u \in \hat{\mathcal{D}}_{t_o}^M$
    \begin{displaymath}
      d \left(
        \check F ( k \tau, t_o+\tau) \circ \check F(\tau,t_o) u, 
        \check F \left( (k+1) \tau, t_o \right)u
      \right)
      \leq
      C \, k\, \tau^2
    \end{displaymath}
    whenever $k\in \naturali$, $(k+1)\tau, \tau \in [0,T-t_o]$;
  \item \label{it:second} a positive constant $L$ such that
    \begin{displaymath}
      d \left( \check F^\epsilon (t,t_o) u, \check F^\epsilon (t,t_o) w \right)
      \leq L \cdot d(u,w)
    \end{displaymath}
    whenever $\epsilon \in \left]0, \delta\right]$, $u, w \in
    \hat{\mathcal{D}}_{t_o}^M$, $t\geq 0$ and $t_o,t_o+t\in [0,T]$.
  \end{enumerate}
\end{proposition}

\noindent Note that~\emph{\ref{it:first}.} states
that~\cite[1.~in~Theorem~2.6]{ColomboGuerra4} is satisfied with
$\omega(t) = Ct$.

\begin{proof}
  To prove~\emph{\ref{it:first}}., the key property
  is~\emph{\ref{it:Lipschitz})} in Theorem~\ref{thm:SRS}, see
  also~\cite[Proposition~4.9]{ColomboGuerraHertySachers}.
  \begin{eqnarray*}
    & &
    \check F ( k \tau, t_o+\tau) \circ \check F(\tau,t_o) u
    -
    \check F \left( (k+1) \tau, t_o \right)u
    \\
    & = &
    P(k\tau,t_o+\tau) 
    \left( 
      P(\tau,t_o)u + \tau G \left(P(\tau,t_o)u\right)
      \chi_{\left[\gamma(t_o+\tau), +\infty\right[}
    \right)
    \\
    & &
    +
    k\tau \, G \left(
      P(k\tau,t_o+\tau) 
      \left( 
        P(\tau,t_o)u + \tau G \left(P(\tau,t_o)u\right)
        \chi_{\left[\gamma(t_o+\tau), +\infty\right[}
      \right)
    \right)
    \cdot
    \\
    & &
    \quad
    \cdot
    \chi_{\left[\gamma(t_o+(k+1)\tau), +\infty\right[} 
    \\
    & &
    -
    P\left((k+1)\tau,t_o\right)u 
    \\
    & &
    - 
    (k+1)\tau \, G \left( P \left((k+1)\tau,t_o\right) u \right)
    \chi_{\left[\gamma(t_o+(k+1)\tau), +\infty\right[}
    \\
    & = &
    P(k\tau,t_o+\tau) 
    \left( 
      P(\tau,t_o)u + \tau G \left(P(\tau,t_o)u\right)
      \chi_{\left[\gamma(t_o+\tau), +\infty\right[}
    \right)
    \\
    & &
    -
    P(k\tau,t_o+\tau) \circ P(\tau,t_o)u
    - 
    \tau \, G \left( P \left((k+1)\tau,t_o\right) u \right)
    \chi_{\left[\gamma(t_o+(k+1)\tau), +\infty\right[}
    \\
    & &
    +
    k\tau
    \bigl[
    G \left(
      P(k\tau,t_o+\tau) 
      \left( 
        P(\tau,t_o)u + \tau G \left(P(\tau,t_o)u\right)
        \chi_{\left[\gamma(t_o+\tau), +\infty\right[}
      \right)
    \right)
    \\
    & &
    \quad
    - 
    G \left( P(k\tau,t_o+\tau) \circ P(\tau,t_o)u \right)
    \bigr]
    \chi_{\left[\gamma(t_o+(k+1)\tau), +\infty\right[} \,.
  \end{eqnarray*}
  Using~\emph{\ref{it:Lipschitz})} in Theorem~\ref{thm:SRS} in the
  first two lines with $t=t'=k\tau$, $t_o=t_o'$ for $t_o+\tau$,
  $v=P(\tau,t_o)u$, $\omega = \tau G \chi$ and in the latter two
  lines~\textbf{(G)}, ~\emph{\ref{it:Lipschitz})} in
  Theorem~\ref{thm:SRS} with $\omega=0$. We thus get
  \begin{eqnarray*}
    & &
    d \left(
      \check F ( k \tau, t_o+\tau) \circ \check F(\tau,t_o) u, 
      \check F \left( (k+1) \tau, t_o \right)u
    \right)
    \\
    & = &
    \norma{
      \check F ( k \tau, t_o+\tau) \circ \check F(\tau,t_o) u
      -
      \check F \left( (k+1) \tau, t_o \right)u
    }_{\L1}
    \\
    & \leq &
    C \tau \Biggl\|
    G \left(P(\tau,t_o)u\right) \chi_{\left[\gamma(t_o+\tau),
        +\infty\right[}
    \\
    & &
    \qquad
    -
    G \left( P \left((k+1)\tau,t_o \right) u \right)
    \chi_{\left[\gamma(t_o+(k+1)\tau), +\infty\right[}    
    \Biggr\|_{\L1}
    \\
    & &
    +
    C \, k\tau^2 \, \norma{G\left(P(\tau,t_o)u\right)}_{\L1}
    \\
    & \leq &
    C\tau \left( 
      k\tau \, \norma{G}_{\L\infty} \, \norma{\dot\gamma}_{\L\infty}
      +
      C \,k\tau
    \right)
    +
    C \left(\norma{G(0)}_{\L1} + 1+ M\right) k\tau^2
    \\
    & \leq &
    C \, (1+M) \, k\tau^2 \,.
  \end{eqnarray*}
  The bound~\emph{\ref{it:second}.}~is a direct consequence of the
  equivalence~(\ref{eq:equiv}) and~(\ref{eq:Bold}) in
  Lemma~\ref{lem:Stime}, see
  also~\cite[Proposition~4.9]{ColomboGuerraHertySachers}
  and~\cite[formula~(3.1)]{ColomboGuerra}.
\end{proof}

\begin{proofof}{\emph{\ref{it:TV})}, \emph{\ref{it:Process})} and \emph{\ref{it:Tangent})} in Theorem~\ref{thm:main}}
  By~\cite[Theorem~2.5]{ColomboGuerra4}, for any $M$, the local flow
  $\check F$ generates a Lipschitz process $\hat P$ on
  $\mathcal{D}_t^M$. By the characterization of $\hat P$ as limit of
  Euler polygonals, it follows that $\hat P$ is uniquely defined on
  all
  \begin{displaymath}
    \hat{\mathcal{D}}_t
    =
    \bigcup_{M>0} \hat{\mathcal{D}}_t^M
    =
    \left\{
      u \in \L1(\reali;\Omega) \colon
      \begin{array}{l}
        u(x) = 0 \mbox{ for all } x < \gamma(t)
        \\
        \boldsymbol{\Upsilon}_t(u) \leq \delta - C(T-t)
      \end{array}
    \right\} \,.
  \end{displaymath}
  Hence, $\hat P$ satisfies~\emph{\ref{it:Process})} in
  Theorem~\ref{thm:main} and~\emph{\ref{it:TV})} holds.

  To prove~\emph{\ref{it:Tangent})}, note that $\frac{1}{t}
  \norma{\hat F(t.t_o)u - \check F(t,t_o)u}_{\L1} \to 0$ as $t\to 0$,
  for all $u \in \hat {\mathcal{D}}_{t_o}$ and
  apply~\cite[\emph{c)}~in Theorem~2.5]{ColomboGuerra4}.
\end{proofof}

For any $N \in \naturali$, define the operator $\Pi_N \colon
\L1(\reali;\reali^n) \to \PC(\reali;\reali^n)$ by
\begin{displaymath}
  \Pi_N (u)
  =
  N \sum_{k=-1-N^2}^{-1+N^2} \int_{k/N}^{(k+1)/N} u(\xi) \, d\xi \,
  \chi_{\strut \left]k/N, (k+1)/N \right]} \,.
\end{displaymath}

\begin{lemma}\label{lemmaconvergence}
  $\Pi_N$ is a linear operator with norm $1$. Moreover, $\tv \left(
    \Pi_N u \right) \leq 2 \tv(u)$ and for all $u \in
  \L1(\reali;\reali^n) \cap \BV(\reali;\reali^n)$, $\Pi_Nu \to u$ in
  $\L1 (\reali;\reali^n)$.
\end{lemma}

For the proof, see~\cite[Lemma~3.4]{ColomboGuerra}.  \medskip

\begin{proofof}{Proposition~\ref{prop:curvesSource}}
  Set for simplicity $t_o=0$. Let $\epsilon, \tilde \epsilon > 0$ and
  $N \in \naturali$ be fixed. Consider an $\tilde\epsilon$-approximate
  wave front tracking solution $u^{\epsilon, \tilde\epsilon, N} =
  u^{\epsilon, \tilde\epsilon, N} (t,x)$ to~(\ref{eq:CPHCL}) on the
  time interval $\left[0,\epsilon\right[$. Define it at time
  $t=\epsilon$ setting
  \begin{displaymath}
    u^{\epsilon, \tilde\epsilon, N} (\epsilon, x)
    =
    u^{\epsilon, \tilde\epsilon, N} (\epsilon -,x) 
    + 
    \epsilon \chi_{\strut\left[\gamma(\epsilon),+\infty\right[}(x)
    \left(
      \Pi_N
      G \left(u^{\epsilon, \tilde\epsilon, N}(\epsilon-) \right)
    \right)
    (x) \,.
  \end{displaymath}
  Extend $u^{\epsilon, \tilde\epsilon, N}$ recursively on
  $[0,T]$. Note that
  \begin{displaymath}
    \lim_{\tilde \epsilon \to 0} \lim_{N \to +\infty} 
    u^{\epsilon, \tilde\epsilon, N}(t)
    =
    \check F^\epsilon (t,0) u_o
  \end{displaymath}
  where $\check F^\epsilon$ is defined in~(\ref{eq:polygonal}) and
  $u_o$ is the initial datum in~(\ref{eq:BL}). Note that this is the
  usual operator splitting algorithm.

  Given any curve $\tilde\ell$--non--characteristic curve $\Gamma$
  with support in $\mathbb{D}_\gamma$, define
  \begin{equation}
    \label{eq:Xi}
    \!\!\!
    \begin{array}{rcl}
      \displaystyle
      \Xi^\epsilon(t)
      & = &
      \displaystyle
      \check K
      \left(
        \sum_{x \geq \Gamma(t)} \sum_{i=1}^{\tilde\ell} \modulo{\sigma_{x,i}}
        +
        \sum_{\gamma(t) \leq x \leq \Gamma(t)} \sum_{i=\tilde\ell+1}^{n+1}
        \modulo{\sigma_{x,i}}
        + 
        \hat K \, \Upsilon^\epsilon(t)
      \right)
      \\[15pt]
      & &
      \displaystyle
      +
      \tv \left( 
        u^{\epsilon, \tilde\epsilon,N} \left(\cdot, \Gamma(\cdot)\right);
        [0,t]
      \right)
    \end{array}
  \end{equation}
  for suitable positive $\hat K, \check K$.

  Computations similar to those above allow to prove that
  $\Xi^\epsilon(t+) \leq \Xi^\epsilon (t-)$ for all $t \not \in
  \epsilon \naturali$. Indeed, when a wave crosses $\Gamma$, the
  increase in $\tv(u^{\epsilon, \bar\epsilon, n})$ is compensated by
  the decrease in the first term on the right hand side
  of~(\ref{eq:Xi}).

  At times $t \in \epsilon \naturali$, $\Xi^\epsilon(t)+ \leq
  \Xi^\epsilon(t-) + C \epsilon$. Therefore, for all $t \in [0,T]$,
  $\Xi^\epsilon(t) \leq \Xi^\epsilon(0) + C t$. By~(\ref{eq:Xi}), we
  get that there exists a $C$ dependent only on $u_o$ and $\tilde\ell$
  such that
  \begin{displaymath}
    \tv \left( 
      u^{\epsilon, \tilde\epsilon,N} \left(\cdot, \Gamma(\cdot)\right);
      [0,t]
    \right)
    \leq C \,.  
  \end{displaymath}
  Consider now two $\tilde\ell-$non-characteristic curves
  $\Gamma_1,\Gamma_2$ with support in $\mathbb{D}_\gamma$. The same
  steps in the proof of Lemma~\ref{lem:curves} lead to
  \begin{displaymath}
    \int_0^T 
    \norma{
      u^{\epsilon,\tilde\epsilon,N}\left(t,\Gamma_1(t)+\right)
      -
      u^{\epsilon,\tilde\epsilon,N}\left(t,\Gamma_2(t)+\right)}
    \, dt
    \leq
    \frac{C}{c} \cdot \norma{\Gamma_2 - \Gamma_1}_{\C0} \,.
  \end{displaymath}
  Let now $\tilde\epsilon \to 0$ and $N \to +\infty$, with the same
  technique of Proposition~\ref{prop:curves} we obtain
  \begin{equation}
    \label{eq:checkF}
    \!\!\!
    \int_0^T 
    \norma{
      \left[ \check F^\epsilon(t,0) u\right] \!\!\left(\Gamma_1(t)+\right)
      -
      \left[ \check F^\epsilon(t,0) u\right] \!\!\left(\Gamma_2(t)+\right)}
    dt
    \leq
    \frac{C}{c} \norma{\Gamma_2 - \Gamma_1}_{\C0}
    \!\!\!
  \end{equation}
  with a constant $C$ that now depends also on $T$ and on $L_2$
  in~\textbf{(G)}. Let now also $\epsilon \to 0$ and, by the
  $\Lloc1(\reali^2;\Omega)$ convergence of the Euler polygonals,
  obtain as in Proposition~\ref{prop:curves} that
  \begin{equation}
    \label{eq:PLip}
    \!\!\!
    \int_0^T 
    \norma{
      \left( \hat P(t,0) u\right) \left(\Gamma_1(t)+\right)
      -
      \left( \hat P(t,0) u\right) \left(\Gamma_2(t)+\right)}
    dt
    \leq
    \frac{C}{c} \norma{\Gamma_2 - \Gamma_1}_{\C0}
    \!\!
  \end{equation}
  completing the proof.
\end{proofof}

\begin{proofof}{\emph{\ref{it:Hard})} and~\emph{\ref{it:bdr})} in
    Theorem~\ref{thm:main}}
  The Lipschitz continuity upon the initial data is a consequence
  of~\cite[\emph{b)}~in Theorem~2.5]{ColomboGuerra4}, thanks to
  Proposition~\ref{prop:yes}. The dependence of the Lipschitz constant
  for the variable $t$ on the $\L1$ norm of the initial data is shown
  in Proposition~\ref{prop:M}.

  The Lipschitz conditions~(\ref{eq:checkF}) and~(\ref{eq:PLip}) allow
  to prove the $\L1\left([0,T];\Omega\right)$ convergence of the
  traces as in Lemma~\ref{lemma:lemma46}:
  \begin{displaymath}
    \left(\check F^\epsilon (\cdot,0) u\right) \left(\Gamma(\cdot)+ \right)
    \to
    \left( P(\cdot,0) u\right) \left(\Gamma(\cdot)+ \right)
    \quad \mbox{ in } \quad \L1 ([0,T];\Omega) \,.
  \end{displaymath}
  The map $t \to \check F^\epsilon(t,0)u$ satisfies for a.e.~$t$ the
  boundary condition, hence the same does the solution $t \to
  P(t,0)u$, proving~2.~in Definition~\ref{def:sol}. Condition~3.~in
  the same definition is proved using the tangency
  condition~\ref{it:tangent}), as
  in~\cite[Corollary~3.14]{ColomboGuerra}.

  We are left to prove the Lipschitz dependence from the boundary and
  the boundary data.  To this aim, introduce two boundaries $\gamma$
  and $\bar \gamma$, with $\gamma \leq \bar\gamma$ and boundary data
  $g$, $\bar g$. Let $\hat{\mathcal{D}}_t$,
  ${\bar{\hat{\mathcal{D}}}}_t$, $\hat P^{g,\gamma}(t,t_o)$ and $\hat
  P^{\bar g, \bar\gamma}(t,t_o)$ the corresponding domains and
  processes.  We need to prove that for any
  $u\in\hat{\mathcal{D}}_0\cap{\bar{\hat{\mathcal{D}}}}_0$ (and
  therefore $u(x)=0$ for $x\leq \bar\gamma(0)$):
  \begin{displaymath}
    \norma{
      \hat P^{g,\gamma}(t,0)u - 
      \hat P^{\bar g,\bar\gamma}(t,0)u 
    }_{\L1(\reali)}
    \leq
    C\!
    \left[
      \norma{\gamma-\bar\gamma}_{\C0([0,t])}
      \! + \!\!
      \int_0^t \! \norma{g(\tau)-\bar g(\tau)} \, d\tau
    \right].
  \end{displaymath}
  Note first that
  \begin{displaymath}
    \begin{split}
      &\norma{ \hat P^{g,\gamma}(t,0)u - \hat P^{\bar g,
          \bar\gamma}(t,0)u }_{\L1(\reali)}
      \\
      &\qquad\qquad\leq C \, \norma{\gamma-\bar\gamma}_{\C0([0,t])} +
      \norma{ \hat P^{g,\gamma}(t,0)u - \hat P^{\bar g,
          \bar\gamma}(t,0)u }_{\L1(I_t)}
    \end{split}
  \end{displaymath}
  where $I_t=\left[\bar\gamma(t),+\infty\right[$.  Hence, we consider
  below only the latter term in the right hand side above. Introduce
  the linear projector $\pi_t v = v \, \chi_{\strut{I_t}}$ and denote
  $w(\tau) = \hat P^{g,\gamma}(\tau,0) u$. Then,
  applying~\cite[Theorem~2.9]{BressanLectureNotes} to the process
  $\hat P^{\bar g, \bar\gamma}$ and to the Lipschitz curve $\tau \to
  \pi_\tau w(\tau)$, using the tangency condition, we compute
  \begin{equation*}
    \begin{split}
      & \norma{ \hat P^{g,\gamma}(t,0)u - \hat P^{\bar g,
          \bar\gamma}(t,0) u }_{\L1(I_t)}
      \\
      &\leq L \int_0^t \liminf_{\epsilon \to 0} \frac{\norma{
          \pi_{\tau+\epsilon}w(\tau+\epsilon) - \hat P^{\bar g,\bar
            \gamma}(\epsilon,\tau) \left(\pi_\tau w(\tau)\right)
        }_{\L1(I_{\tau+\epsilon})}}{\epsilon} d\tau
      \\
      &\leq L \int_0^t \liminf_{\epsilon \to 0} \frac{\norma{ \hat
          P^{g,\gamma}(\epsilon,\tau)w(\tau) - \hat P^{\bar g,\bar
            \gamma}(\epsilon,\tau) \left(\pi_\tau w(\tau)\right)
        }_{\L1(I_{\tau+\epsilon})}}{\epsilon} d\tau
      \\
      &\leq L \int_0^t \liminf_{\epsilon \to 0} \frac{\norma{
          P^{g,\gamma}(\epsilon,\tau)\left(w(\tau)\right) - P^{\bar
            g,\bar \gamma}(\epsilon,\tau) \left( \pi_\tau w(\tau)
          \right) }_{\L1(I_{\tau+\epsilon})}}{\epsilon} d\tau
      \\
      & + L \int_0^t \norma{ G\left(w(\tau)\right) - G\left(\pi_\tau
          w(\tau) \right) }_{\L1({\reali})} d\tau \,.
    \end{split}
  \end{equation*}
  For the term deriving from the source, we use the $\L1$ Lipschitz
  continuity of $G$ to estimate:
  \begin{eqnarray*}
    \int_0^t \norma{
      G\left(w(\tau)\right)
      - 
      G\left(\pi_\tau w(\tau) \right)
    }_{\L1({\reali})}
    d\tau
    & \leq &
    C \int_0^t \norma{
      w(\tau)-\left( \pi_\tau w(\tau) \right) 
    }_{\L1(\reali)}
    d\tau\\
    & \leq &
    C \, T \, \norma{\gamma - \bar \gamma}_{\C0([0,t])}.
  \end{eqnarray*}
  Concerning the other term, denote by $ F^{g_o,\gamma}(t,t_o)u $ the
  tangent vector defined in~(\ref{eq:local}). Here, we explicitly
  denote the dependence of the tangent vector on the curve $\gamma$
  and on the pointwise boundary data $g_o=g(t_o)$. By
  \ref{it:tangent}) in Theorem \ref{thm:SRS}, the curve $\eta \to
  P^{\bar g,\bar \gamma}(\eta,\tau) \left( \pi_\tau w(\tau) \right)$
  is first order tangent to $\eta \to F^{\bar g(\tau),\bar
    \gamma}(\eta,\tau) \left( \pi_\tau w(\tau) \right)$, while $\eta
  \to P^{ g,\gamma}(\eta,\tau) w(\tau)$ is first order tangent to
  $\eta \to F^{ g(\tau),\gamma}(\eta,\tau) \left(w(\tau)
  \right)$. Because of the finite propagation speed, the two tangent
  vectors coincide in the interval
  $[\bar\gamma(\tau)+\eta\hat\lambda,+\infty[$. Therefore,
  \begin{equation*}
    \begin{split}
      & \int_0^t \liminf_{\epsilon \to 0} \frac{\norma{ P^{
            g,\gamma}(\epsilon,\tau) w(\tau) - P^{\bar g,\bar
            \gamma}(\epsilon,\tau) \left( \pi_\tau w(\tau) \right)
        }_{\L1(I_{\tau+\epsilon})}}{\epsilon} d\tau
      \\
      &= \int_0^t \liminf_{\epsilon \to 0} \frac{\norma{
          F^{g(\tau),\gamma}(\epsilon,\tau) w(\tau) - F^{\bar
            g(\tau),\bar \gamma}(\epsilon,\tau) \left( \pi_\tau
            w(\tau) \right) }_{\L1(I_{\tau+\epsilon})}}{\epsilon}
      d\tau
      \\
      &= \int_0^t \liminf_{\epsilon \to 0} \frac{1}{\epsilon}
      \int_{\bar\gamma(\tau+\epsilon)}^{\bar\gamma(\tau)+\epsilon\hat\lambda}
      \bigg\| \left(F^{g(\tau),\gamma}(\epsilon,\tau)
        w(\tau)\right)(x)\\
      &\qquad\qquad - \left(F^{\bar g(\tau),\bar
          \gamma}(\epsilon,\tau) \left( \pi_\tau w(\tau)
        \right)\right)(x)\bigg\| \,dx d\tau \,.
    \end{split}
  \end{equation*}
  Referring to Lemma~\ref{lem:RP}, introduce the
  quantities
  \begin{eqnarray*}
    w^\tau_r 
    & = & 
    w\left(\tau,\bar\gamma(\tau) \right),
    \\
    b(w^{\bar\sigma,\tau}) 
    & = & 
    \bar g(\tau),
    \\
    w_r^\tau 
    & = & 
    \psi_n(\bar\sigma_n)\circ\ldots\circ\psi_{\ell+1}(\sigma_{\ell+1})
    (w^{\bar\sigma,\tau}).
    \\
    \tilde w^\tau(x)
    & = &
    \begin{cases}
      w(\tau,x) &\text{ for } x\geq \gamma(\tau),
      \\
      w \left(\tau,\gamma(\tau) \right) & \text{ for } x <
      \gamma(\tau),
    \end{cases}
    \\
    \tilde{\bar w}^\tau(x)
    & = &
    \begin{cases}
      w(\tau,x) & \text{ for } x\geq \bar\gamma(\tau),
      \\
      w^{\bar\sigma,\tau} & \text{ for } x < \bar\gamma(\tau)
    \end{cases}
  \end{eqnarray*}
  By formul\ae~(\ref{eq:tilde})--(\ref{eq:local}) and since the
  boundary condition is satisfied for almost all $\tau$, that is
  $b\left( w\left(\tau,\gamma(\tau) \right) \right) = g(\tau)$, one
  has for $x\geq \bar\gamma(\tau+\epsilon)$
  \begin{eqnarray*}
    \left(F^{g(\tau),\gamma}(\epsilon,\tau) 
      w(\tau)\right)(x)
    & = &
    \left({\mathcal{S}}_\epsilon\tilde w^\tau\right)(x)
    \\
    \left(F^{\bar g(\tau),\bar\gamma}(\epsilon,\tau) \left(\pi_\tau 
        w(\tau)\right)\right)(x)
    & = &
    \left({\mathcal{S}}_\epsilon {\tilde{ \bar{w}}}^\tau\right)(x)
  \end{eqnarray*}
  where $\mathcal{S}$ is the purely convective Standard Riemann
  Semigroup without boundary generated by
  $f$~\cite[Definition~9.1]{BressanLectureNotes}.
    
  Denote by $U^\sharp_{\tau}$ and $\bar U^\sharp_{\tau}$ the solutions
  to the two Riemann problems:
  \begin{displaymath}
    \begin{cases}
      u_t+f(u)_x=0
      \\
      u(0,x)=
      \begin{cases}
        \tilde w^\tau(\bar\gamma(\tau)-) & \text{ for }x<0
        \\
        \tilde w^\tau \left( \bar\gamma(\tau) \right) & \text{ for
        }x>0
      \end{cases}
    \end{cases}
    \begin{cases}
      u_t+f(u)_x=0
      \\
      u(0,x)=
      \begin{cases}
        \tilde{\bar w}^\tau(\bar\gamma(\tau)-) &\text{ for }x<0
        \\
        \tilde{\bar w}^\tau \left(\bar\gamma(\tau) \right) & \text{
          for }x>0
      \end{cases}
    \end{cases}
  \end{displaymath}
  Formula~\cite[(9.16)]{BressanLectureNotes} implies that
  \begin{equation*}
    \begin{split}
      &\int_0^t \liminf_{\epsilon \to 0} \frac{1}{\epsilon}
      \int_{\bar\gamma(\tau+\epsilon)}^{\bar\gamma(\tau)+
        \epsilon\hat\lambda} \bigg\|
      \left({\mathcal{S}}_\epsilon\tilde w^\tau\right)(x) -
      \left({\mathcal{S}}_\epsilon\tilde w^\tau\right)(x) \bigg\| \,dx
      d\tau\\
      &\qquad\leq \int_0^t \liminf_{\epsilon \to 0} \frac{1}{\epsilon}
      \int_{\bar\gamma(\tau+\epsilon)-\bar\gamma(\tau)}^{
        \epsilon\hat\lambda} \left\| U^\sharp_{\tau}(\epsilon,x)-\bar
        U^\sharp_{\tau}(\epsilon,x) \right\| \,dx d\tau
    \end{split}
  \end{equation*}
  By Remark \ref{finalremark}, for almost all $\tau$ such that
  $\gamma(\tau) < \bar\gamma(\tau)$ one has $ \tilde
  w^\tau(\bar\gamma(\tau)-) = \tilde w^\tau(\bar\gamma(\tau)-) =
  w_r^\tau$, therefore $U^\sharp_{\tau}(\epsilon,x)\equiv w_r$. While
  for almost all $\tau$ such that $\gamma(\tau) = \bar\gamma(\tau)$,
  the boundary condition implies $\tilde w^\tau(\bar\gamma(\tau)-) =
  \tilde w^\tau \left( \gamma(\tau) \right) = w_r^\tau$ therefore we
  have again $U^\sharp_{\tau}(\epsilon,x)\equiv w_r^\tau$. We compute,
  for almost all $\tau$
  \begin{equation*}
    \begin{split}
      &\left\| U^\sharp_{\tau}(\epsilon,x)-\bar
        U^\sharp_{\tau}(\epsilon,x) \right\|\le
      C\left\|E^{\sigma}_b\left(w_r^\tau,\bar
          g(\tau)\right)\right\|\\
      &\qquad =C\left\|E^{\sigma}_b\left(w_r^\tau,\bar
          g(\tau)\right)-E^{\sigma}_b\left(w_r^\tau,b(w_r^\tau)\right)\right\|\le
      C\left\|\bar g(\tau)-b\left(w_r^\tau\right)\right\|.
    \end{split}
  \end{equation*}
  Finally we compute, using Proposition \ref{prop:curvesSource},
  \begin{equation*}
    \begin{split}
      &\int_0^t \liminf_{\epsilon \to 0} \frac{1}{\epsilon}
      \int_{\bar\gamma(\tau+\epsilon)-\bar\gamma(\tau)}^{
        \epsilon\hat\lambda} \left\| U^\sharp_{\tau}(\epsilon,x)-\bar
        U^\sharp_{\tau}(\epsilon,x) \right\| \,dx d\tau
      \\
      &\leq C \int_0^t\left\|\bar
        g(\tau)-b\left(w_r^\tau\right)\right\|\, d\tau
      \\
      &\leq C \int_0^t\left\|\bar g(\tau)-g(\tau)\right\|\, d\tau + C
      \int_0^t \left\| b\left(w \left(\tau,\gamma(\tau) \right)
        \right) - b\left(w \left(\tau,\bar\gamma(\tau) \right) \right)
      \right\|\, d\tau
      \\
      & \leq C \int_0^t\left\|\bar g(\tau)-g(\tau)\right\|\, d\tau +
      \left\|\gamma-\bar\gamma\right\|_{\C0}.
    \end{split}
  \end{equation*}
  The general case of two non ordered curves follows immediately by
  the triangle inequality.
\end{proofof}

{\small{

    \bibliographystyle{abbrv}

    \bibliography{dege}

\def\cprime{$'$}
\begin{thebibliography}{10}

\bibitem{Amadori1}
D.~Amadori.
\newblock Initial-boundary value problems for nonlinear systems of conservation
  laws.
\newblock {\em NoDEA Nonlinear Differential Equations Appl.}, 4(1):1--42, 1997.

\bibitem{AmadoriColombo1}
D.~Amadori and R.~M. Colombo.
\newblock Continuous dependence for $2\times 2$ conservation laws with
  boundary.
\newblock {\em J. Differential Equations}, 138(2):229--266, 1997.

\bibitem{AmadoriGuerra2002}
D.~Amadori and G.~Guerra.
\newblock Uniqueness and continuous dependence for systems of balance laws with
  dissipation.
\newblock {\em Nonlinear Anal.}, 49(7, Ser. A: Theory Methods):987--1014, 2002.

\bibitem{BaitiBressan2}
P.~Baiti and A.~Bressan.
\newblock Lower semicontinuity of weighted path lengths in {B}{V}.
\newblock In F.~Colombini and N.~Lerner, editors, {\em Geometrical Optics and
  Related Topics}, pages 31--58. Birkh\"auser, Boston, 1997.

\bibitem{BressanCauchy}
A.~Bressan.
\newblock On the {C}auchy problem for systems of conservation laws.
\newblock In {\em Actes du 29\`eme Congr\`es d'Analyse Num\'erique: CANum'97
  (Larnas, 1997)}, pages 23--36 (electronic). Soc. Math. Appl. Indust., Paris,
  1998.

\bibitem{BressanLectureNotes}
A.~Bressan.
\newblock {\em Hyperbolic systems of conservation laws}, volume~20 of {\em
  Oxford Lecture Series in Mathematics and its Applications}.
\newblock Oxford University Press, Oxford, 2000.
\newblock The one-dimensional Cauchy problem.

\bibitem{ColomboGuerra}
R.~M. Colombo and G.~Guerra.
\newblock Hyperbolic balance laws with a non local source.
\newblock {\em Communications in Partial Differential Equations},
  32(12):1917--1939, 2007.

\bibitem{ColomboGuerra3}
R.~M. Colombo and G.~Guerra.
\newblock Hyperbolic balance laws with a dissipative non local source.
\newblock {\em Commun. Pure Appl. Anal.}, 7(5):1077--1090, 2008.

\bibitem{ColomboGuerra2}
R.~M. Colombo and G.~Guerra.
\newblock On the stability functional for conservation laws.
\newblock {\em Nonlinear Anal.}, 69(5-6):1581--1598, 2008.

\bibitem{ColomboGuerra4}
R.~M. Colombo and G.~Guerra.
\newblock Differential equations in metric spaces with applications.
\newblock {\em Discrete and Continuous Dynamical Systems, Series A}, To appear.

\bibitem{ColomboGuerraHertySachers}
R.~M. Colombo, G.~Guerra, M.~Herty, and V.~Sachers.
\newblock Modeling and optimal control of networks of pipes and canals.
\newblock {\em arxiv:0802.3613v1}, 2008.

\bibitem{ColomboRosini4}
R.~M. Colombo and M.~D. Rosini.
\newblock Well posedness of balance laws with non-characteristic boundary.
\newblock {\em Bollettino U.M.I.}, To appear.

\bibitem{DafermosBook}
C.~M. Dafermos.
\newblock {\em Hyperbolic conservation laws in continuum physics}, volume 325
  of {\em Grundlehren der Mathematischen Wissenschaften [Fundamental Principles
  of Mathematical Sciences]}.
\newblock Springer-Verlag, Berlin, second edition, 2005.

\bibitem{DafermosHsiao}
C.~M. Dafermos and L.~Hsiao.
\newblock Hyperbolic systems and balance laws with inhomogeneity and
  dissipation.
\newblock {\em Indiana Univ. Math. J.}, 31(4):471--491, 1982.

\bibitem{DonadelloMarson}
C.~Donadello and A.~Marson.
\newblock Stability of front tracking solutions to the initial and boundary
  value problem for systems of conservation laws.
\newblock {\em NoDEA Nonlinear Differential Equations Appl.}, 14(5-6):569--592,
  2007.

\bibitem{DuboisLefloch}
F.~Dubois and P.~LeFloch.
\newblock Boundary conditions for nonlinear hyperbolic systems of conservation
  laws.
\newblock {\em J. Differential Equations}, 71(1):93--122, 1988.

\bibitem{Goodman}
J.~Goodman.
\newblock {\em Initial Boundary Value Problems for Hyperbolic Systems of
  Conservation Laws}.
\newblock PhD thesis, California University, 1982.

\bibitem{TPLiuQuasilinear}
T.~P. Liu.
\newblock Quasilinear hyperbolic systems.
\newblock {\em Comm. Math. Phys.}, 68(2):141--172, 1979.

\end{thebibliography}

  }}

\end{document}